\documentclass[11pt,reqno]{amsart}
\usepackage{amssymb,amsthm,amsmath,epsfig,latexsym,hyperref,enumerate}
\usepackage{calc}
\usepackage{times}
\usepackage[all]{xy}
\usepackage{tikz}
\usetikzlibrary{matrix}
\usepackage[width=5.65in,height=8.0in,centering]{geometry}

\newcommand{\K}{{\Bbbk}}   
\newcommand{\A}{{\mathcal A}}  

\newcommand{\proj}{{\mathbb{P}}}
\newcommand{\poin}{Poincar\'{e} }

\newcommand{\cI}{{\mathcal I}}
\newcommand{\surjects}{\twoheadrightarrow}

\newcommand{\rar}{\rightarrow}
\newcommand{\lar}{\longrightarrow}
\newcommand{\llar}{-\kern-5pt-\kern-5pt\longrightarrow}
\DeclareMathOperator{\OT}{OT}  

\newcommand{\B}{{\mathcal B}}

\newcommand{\rees}{{R[It]}}
\newcommand{\fiber}{{\mathcal{F}}(I)}

\newcommand{\recip}{{\mathcal{L}}^{-1}_{\A}}

\newcommand{\fm}{\mathfrak m}
\newcommand{\xx}{\mathbf x}
\newtheorem{defn0}{Definition}[section]
\newtheorem{prop0}[defn0]{Proposition}
\newtheorem{conj0}[defn0]{Conjecture}
\newtheorem{thm0}[defn0]{Theorem}
\newtheorem{lem0}[defn0]{Lemma}
\newtheorem{corollary0}[defn0]{Corollary}
\newtheorem{example0}[defn0]{Example}
\newtheorem{remark0}[defn0]{Remark}
\newtheorem{question0}[defn0]{Question}

\newenvironment{prop}{\begin{prop0}}{\end{prop0}}

\newenvironment{thm}{\begin{thm0}}{\end{thm0}}
\newenvironment{lem}{\begin{lem0}}{\end{lem0}}
\newenvironment{cor}{\begin{corollary0}}{\end{corollary0}}
\newenvironment{exm}{\begin{example0}\rm}{\end{example0}}
\newenvironment{rem}{\begin{remark0}\rm}{\end{remark0}}

\begin{document}



\title{A blowup algebra of hyperplane arrangements}

\author{Mehdi Garrousian, Aron Simis and \c Stefan O. Toh\u aneanu}

\subjclass[2010]{Primary 13A30, 14N20; Secondary 13C14, 13D02, 13D05}
\keywords{Rees algebra, special fiber algebra, Orlik-Terao algebra, Cohen-Macaulay.\\
\indent The second author has been partially supported by a CNPq grant (302298/2014-2) and by a CAPES-PVNS Fellowship (5742201241/2016); he is currently holding a one year Senior Visiting Professorship (2016/08929-7) at the Institute for Mathematical and Computer Sciences, University of S\~{a}o Paulo, S\~{a}o Carlos, Brazil.\\
\indent Garrousian's Address: Department of Mathematics, University of Western Ontario, London, ON N6A 5B7, Canada, Email: mgarrous@alumni.uwo.ca.\\
\indent Simis' Address: Departamento de Matem\'{a}tica, Centro de Ci\^{e}ncias Exatas e da Natureza, Universidade Federal de Pernambuco, 50740-560, Recife, Pernambuco, Brazil, Email: aron@dmat.ufpe.br.\\
\indent Tohaneanu's Address: Department of Mathematics, University of Idaho, Moscow, Idaho 83844-1103, USA, Email: tohaneanu@uidaho.edu, Phone: 208-885-6234, Fax: 208-885-5843.}

\begin{abstract}
\noindent It is shown that the Orlik-Terao algebra is graded isomorphic to the special fiber of the ideal $I$ generated by the $(n-1)$-fold products of the members of a central arrangement of size $n$. This momentum is carried over to the Rees algebra (blowup) of $I$ and it is shown that this algebra is of fiber-type and Cohen-Macaulay. It follows by a result of Simis-Vasconcelos that the special fiber of $I$ is Cohen-Macaulay, thus giving another proof of a result of Proudfoot-Speyer about the Cohen-Macauleyness of the Orlik-Terao algebra.
\end{abstract}
\maketitle


\section*{Introduction}

The central theme of this paper is to study the ideal theoretic aspects of the blowup of a projective space along a certain scheme of codimension $2$. To be more precise, let $\A=\{\ker (\ell_1),\dots,\ker( \ell_n)\}$ be an arrangement of hyperplanes in $\mathbb{P}^{k-1}$ and consider the closure of the graph of the following rational map  
\[
 \mathbb{P}^{k-1} \dashrightarrow  {\mathbb{P}}^{n-1}, \quad x\mapsto (1/\ell_1(x):\cdots :1/\ell_n(x)).
\]
Rewriting the coordinates of the map as forms of the same positive degree in the source $\mathbb{P}^{k-1}={\rm Proj}(\K[x_1,\ldots,x_k])$, we are led to consider the corresponding graded $\K[x_1,\ldots,x_k]$-algebra, namely, the Rees algebra of the ideal of $\K[x_1,\ldots,x_k]$ generated by the $(n-1)$-fold products of $\ell_1,\ldots,\ell_n$.

 It is our view that bringing into the related combinatorics a limited universe of gadgets and numerical invariants from commutative algebra may be of help, specially regarding the typical operations with ideals and algebras. This point of view favors at the outset a second look at the celebrated Orlik-Terao algebra $\K[1/\ell_1,\dots, 1/\ell_n]$ which is regarded as a commutative counterpart to the combinatorial Orlik-Solomon algebra. The fact that the former, as observed by some authors, has a model as a finitely generated graded $\K$-subalgebra of a finitely generated purely transcendental extension of the field $\K$, makes it possible to recover it as the homogeneous coordinate ring of the image of a certain rational map.

This is our departing step to naturally introduce other commutative algebras into the picture. As shown in Theorem \ref{ot_specialfiber}, the Orlik-Terao algebra now becomes isomorphic, as graded $\K$-algebra, to the special fiber algebra (also called fiber cone algebra or central algebra) of the ideal $I$ generated by the $(n-1)$-fold products of the members of the arrangement $\A$.
This algebra is in turn defined as a residue algebra of the Rees algebra of $I$, so it is only natural to look at this and related constructions. One of these constructions takes us to the symmetric algebra of $I$, and hence to the syzygies of $I$.
Since $I$ turns out to be a perfect ideal of codimension $2$, its syzygies are rather simple and allow us to further understand these algebras.

As a second result along this line of approach, we show that a presentation ideal of the Rees algebra of $I$ can be generated by the syzygetic relations and the Orlik-Terao ideal (see Theorem \ref{fiber_type}).
This property has been coined {\em fiber type} property in the recent literature of commutative algebra.

The third main result of this work, as an eventual outcome of these methods, is a proof of the Cohen-Macaulay property of the Rees algebra of $I$ (see Theorem \ref{thm:CM}).

The typical argument in the proofs is induction on the size or rank of the arrangements. Here we draw heavily on the operations of deletion and contraction of an arrangement. In particular, we introduce a variant of a multiarrangement that allows repeated linear forms to be tagged with arbitrarily different coefficients.
Then the main breakthrough consists in getting a precise relation between the various ideals or algebras attached  to the original arrangement and those attached to the minors.

One of the important facts about the Orlik-Terao algebra is that it is Cohen-Macaulay, as proven by Proudfoot-Speyer \cite{PrSp}. Using a general result communicated by the second author and Vasconcelos, we recover this result as a consequence of the Cohen-Macaulay property of the Rees algebra. 


\smallskip

The structure of this paper is as follows. The first section is an account of the needed preliminaries from commutative algebra. The second section expands on highlights of the settled literature about the Orlik-Terao ideal as well as a tangential discussion on the so-called non-linear invariants of our ideals such as the reduction number and analytic spread. The third section focuses on the ideal of $(n-1)$-fold products and the associated algebraic constructions. The last section is devoted to the statements and proofs of the main theorems where we draw various results from the previous sections to establish the arguments. 

\section{Ideal theoretic notions and blowup algebras}\label{algebraic}

The \emph{Rees algebra} of an ideal $I$ in a ring $R$ is the $R$-algebra
\[
\mathcal{R}(I):=\bigoplus_{i\geq 0} I^i.
\]
This is a standard $R$-graded algebra with $\mathcal{R}(I)_0=R$, where multiplication is induced by the internal multiplication rule $I^r I^s\subset I^{r+s}$.
One can see that there is a graded isomorphism $R[It]\simeq \mathcal{R}(I)$, where $R[It]$ is the homogeneous $R$-subalgebra of the standard graded polynomial $R[t]$ in one variable over $R$, generated by the elements $at, a\in I$, of degree $1$.

Quite generally, fixing a set of generators of $I$ determines a surjective homomorphism of $R$-algebras from a polynomial ring over $R$ to $R[It]$.
The kernel of such a map is called a {\em presentation ideal} of $R[It]$.
In this generality, even if $R$ is Noetherian (so $I$ is finitely generated) the notion of a presentation ideal is quite loose.

In this work we deal with a special case in which $R=\K[x_1,\dots, x_k]$ is a standard graded polynomial ring over a field $\K$ and
$I=\langle g_1,\ldots, g_n\rangle$ is an ideal generated by forms $ g_1,\dots, g_n$ of the same degree.
Let $T=R[y_1,\ldots,y_n]=\K[x_1,\dots, x_k; y_1,\ldots,y_n]$, a standard bigraded $\K$-algebra with $\deg x_i=(1,0)$ and $\deg y_j=(0,1)$.
Using the given generators to obtain an $R$-algebra homomorphism
$$\varphi:T=R[y_1,\ldots,y_n]\longrightarrow R[It],\, y_i\mapsto g_it,$$
yields a presentation ideal $\mathcal{I}$ which is bihomogeneous in the bigrading of $T$.
Therefore, $R[It]$ acquires the corresponding bigrading.

Changing $\K$-linearly independent sets of generators in the same degree amounts to effecting an invertible $\K$-linear map, so the resulting effect on the corresponding presentation ideal is pretty much under control.
For this reason, we will by abuse talk about {\em the} presentation ideal of $I$ by fixing a particular set of homogeneous generators of $I$ of the same degree.
Occasionally, we may need to bring in a few superfluous generators into a set of minimal generators.

Since the given generators have the same degree they span a linear system defining a rational map
\begin{equation}\label{map}
	\Phi:{\mathbb{P}}^{k-1} \dashrightarrow  {\mathbb{P}}^{n-1},
\end{equation}
by the assignment $x \mapsto  (g_1(x):\dots : g_n(x))$, when some $g_i(x)\neq 0$.

The ideal $I$ is often called the base ideal (to agree with the base scheme) of $\Phi$.
Asking when $\Phi$ is birational onto its image is of interest and we will briefly deal with it as well.
Again note that changing to another set of  generators in the same degree will not change the linear system thereof, defining the same rational map up to a coordinate change at the target.

The Rees algebra brings along other algebras of interest.
In the present setup, one of them is
the \emph{special fiber}  $\fiber:=\rees\otimes_R R/{\fm}\simeq \oplus_{s\geq 0}I^s/{\fm}I^s$, where $\fm=\langle x_1,\ldots,x_k\rangle\subset R$. The Krull dimension of the special fiber $\ell(I):=\dim \fiber$ is called the {\em analytic spread} of $I$.

The analytic spread is a significant notion in the theory of {\em reductions} of ideals.
An ideal $J\subset I$ is said to be a {\em reduction of $I$} if $I^{r+1}=JI^r$ for some $r$. Most notably, this is equivalent to the condition that the natural inclusion $R[Jt]\hookrightarrow R[It]$ is a finite morphism. The smallest such $r$ is the {\em reduction number $r_J(I)$ with respect to $J$}. The {\em reduction number of $I$} is the infimum of all $r_J(I)$ for all minimal reductions $J$ of $I$; this number is denoted by $r(I)$.

Geometrically, the relevance of the special fiber lies in the following result, which
we isolate for easy reference:
\begin{lem}\label{fiber_principle}
Let $\Phi$ be as in {\rm (\ref{map})} and $I$ its base ideal. Then the homogeneous coordinate ring of the image of $\Phi$ is isomorphic to the special fiber $\fiber$ as graded $\K$-algebras.
\end{lem}
To see this, note that the Rees algebra defines a biprojective subvariety of $\proj^{k-1}\times \proj^{n-1}$, namely the closure of the graph of $\Phi$.
Projecting down to the second coordinate recovers the image of $\Phi$.
At the level of coordinate rings this projection corresponds to the inclusion $\K[I_dt]=\K[g_1t,\ldots,g_nt]\subset R[It]$, where $g_1,\ldots, g_n$ are forms of the degree $d$, this inclusion is a split $\K[I_dt]$-module homomorphism with ${\fm}R[It]$ as direct complement. Therefore, one has an isomorphism of $\K$-graded algebras $\K[I_d]\simeq \K[I_dt]\simeq \fiber$.

\medskip

As noted before,  the presentation ideal of $R[It]$
$$\mathcal I=\bigoplus_{(a,b)\in {\mathbb N} \times  {\mathbb N}} {\mathcal I}_{(a,b)},$$
is a bihomogeneous ideal in the standard bigrading of $T$.
Two basic subideals of $\mathcal I$ are $\langle \mathcal I_{(0,-)}\rangle$ and $\langle \mathcal I_{(-,1)}\rangle$.
The two come in as follows.

Consider the natural surjections
\[
\xymatrix{T  \ar@/_1pc/[rr]_{\psi} \ar[r]^{\varphi}  & R[It] \ar[r]^{\otimes_R R/{\mathfrak{m}}} & {\mathcal{F}}(I)
}
\]
where the kernel of the leftmost map is the presentation ideal $\mathcal I$ of $R[It]$.
Then, we have
\[
\fiber \simeq \frac{T}{\ker \psi}\simeq \frac{T}{\langle \ker \varphi|_{(0,-)},\mathfrak{m}\rangle}\simeq\frac{\K[y_1,\ldots,y_n]}{\langle \mathcal I_{(0,-)}\rangle}.
\]
Thus, $\langle \mathcal I_{(0,-)}\rangle$ is the homogeneous defining ideal of the special fiber (or, as explained in Lemma~\ref{fiber_principle}, of the image of the rational map $\Phi$).

As for the second ideal $\langle \mathcal I_{(-,1)}\rangle$, one can see that it coincides with the ideal of $T$  generated by the biforms $s_1y_1+\cdots+s_ny_n\in T$, whenever $(s_1,\ldots,s_n)$ is a syzygy of $g_1,\ldots,g_n$ of certain degree in $R$.
Thinking about the one-sided grading in the $y$'s there is no essential harm in denoting this ideal simply by  $\mathcal I_1$.
Thus, $T/\mathcal I_1$ is a presentation of the symmetric algebra $\mathcal S (I)$ of $I$.
It obtains a natural surjective map of $R$-graded algebras
$$\mathcal S (I)\simeq  T/\mathcal I_1 \surjects T/{\mathcal I} \simeq \mathcal R (I).$$
As a matter of calculation, one can easily show that $\mathcal I=\mathcal I_1:I^{\infty},$ the saturation of $\mathcal I_1$ with respect to $I$.

The ideal $I$ is said to be of {\em linear type} provided $\mathcal I=\mathcal I_1$, i.e., when the above surjecton is injective.
It is said to be of {\em fiber type} if $\mathcal I= \mathcal I_1 + \langle \mathcal I_{(0,-)}\rangle=\langle \mathcal I_1,\, \mathcal I_{(0,-)}\rangle$.

A basic homological obstruction for an ideal to be of linear type is the so-called $G_{\infty}$ condition of Artin-Nagata \cite{ArNa}, also known as the $F_1$ condition \cite{Trento}.
A weaker condition is the so-called $G_s$ condition, for a suitable integer $s$.
All these conditions can be stated in terms of the Fitting ideals of the given ideal or, equivalently, in terms of the various ideals of minors of a syzygy matrix of the ideal.
In this work we will have a chance to use condition $G_k$, where $k=\dim R<\infty$.
Given a free presentation
$$R^m \stackrel{\varphi}{\lar} R^n\lar I \rar 0$$
of an ideal $I\subset R$, the $G_k$ condition for $I$ means that
\begin{equation}\label{G_k in minors}
	{\rm ht}(I_p(\varphi))\geq n-p+1, \; {\rm for }\; p\geq n-k+1,
\end{equation}
where $I_t(\varphi)$ denotes the ideal generated by the $t$-minors of $\varphi$.
Note that nothing is required about the values of $p$ strictly smaller than $n-k+1$ since for such values one has $n-p+1>k=\dim R$, which makes the same bound impossible.

\medskip

A useful method to obtain new generators of $\mathcal I$ from old generators (starting from generators of $\mathcal I_1$) is via Sylvester forms (see \cite[Proposition 2.1]{HoSiVa}), which has classical roots as the name indicates.
It can be defined in quite generality as follows:
let $R:=\K[x_1,\ldots,x_k]$, and let $T:=R[y_1,\ldots,y_n]$ as above. Given $F_1,\ldots,F_s\in \mathcal I$,  let $J$ be the ideal of $R$ generated by all the coefficients of the $F_i$ -- the so-called {\em $R$-content ideal}. Suppose $J=\langle a_1,\ldots,a_q\rangle$, where $a_i$ are forms of the same degree. Then we have the matrix equation:
$$\left[\begin{array}{c} F_1\\ F_2\\\vdots\\F_s
\end{array}\right]={\bf A}\cdot \left[\begin{array}{c} a_1\\ a_2\\\vdots\\a_q
\end{array}\right],$$ where ${\bf A}$ is an $s\times q$ matrix with entries in $T$.

If $q\geq s$ and if the syzygies on $F_i'$s are in $\fm T$, then the determinant of any $s\times s$ minor of ${\bf A}$ is an element of $\mathcal I$. These determinants are called {\em Sylvester forms}.
The main use in this work is to show that the Orlik-Terao ideal is generated by such forms (Proposition~\ref{ot_ideal}).

The last invariant we wish to comment on is the reduction number $r(I)$.
For convenience, we state the following result:
\begin{prop}\label{CMfiber}
With the above notation, suppose that the special fiber $\mathcal{F}(I)$ is Cohen-Macaulay. Then the reduction number $r(I)$ of $I$ coincides with the Castelnuovo-Mumford regularity ${\rm reg}(\mathcal F(I))$ of $\mathcal{F}(I)$.
\end{prop}
\begin{proof}
 By \cite[Proposition 1.85]{Va}, when the special fiber is Cohen-Macaulay, one can read $r(I)$ off the Hilbert series. Write
\[
HS(\mathcal{F}(I),s)=\frac{1+h_1s+h_2s^2+\cdots +h_rs^r}{(1-s)^d},
\]
 with $h_r\neq 0$ and $d=\ell(I)$, the dimension of the fiber (analytic spread). Then, $r(I)=r$.

Since $\mathcal F(I)\simeq S/\langle\mathcal I_{(0,-)}\rangle$, where $S:=\K[y_1,\ldots,y_n]$, we have that $\mathcal F(I)$ has a minimal graded $S$-free resolution of length equal to $m:={\rm ht}\, \langle\mathcal I_{(0,-)}\rangle$, and ${\rm reg}(\mathcal F(I))=\alpha-m$, where $\alpha$ is the largest shift in the minimal graded free resolution, occurring also at the end of this resolution. These last two statements mentioned here come from the Cohen-Macaulayness of $\mathcal F(I)$.

The additivity of Hilbert series under short exact sequences of modules, together with the fact that $\displaystyle HS(S^u(-v),s)=u\frac{s^v}{(1-s)^n}$ gives that $r+m=\alpha=m+{\rm reg}(\mathcal F(I))$, so $r(I)={\rm reg}(\mathcal F(I)).$	
\end{proof}

\section{Hyperplane Arrangements}\label{gens}

Let $\A=\{H_1,\ldots,H_n\}\subset\mathbb P^{k-1}$ be a central hyperplane arrangement of size $n$ and rank $k$. Here $H_i=\ker(\ell_i),i=1,\ldots,n$, where each $\ell_i$ is a linear form in $R:=\K[x_1,\ldots,x_k]$ and $\langle\ell_1,\ldots,\ell_n\rangle=\fm:=\langle x_1,\ldots,x_k\rangle$. From the algebraic viewpoint, there is a natural emphasis on the linear forms $\ell_i$ and the associated ideal theoretic notions.

Deletion and contraction are useful operations upon $\A$.
Fixing an index $1\leq i\leq n$, one introduces two new minor arrangements:
\[
 \A'=\A\setminus\{H_i\}\: {\rm (deletion)}, \quad \A''=\A'\cap H_i:=\{H_j\cap H_i\,|\,1\leq j\leq n\, , j\neq i\} \: {\rm (contraction)}.
\]
Clearly, $\A'$ is a subarrangement of $\A$ of size $n-1$ and rank at most $k$, while $\A''$ is an arrangement of size $\leq n-1$ and rank $k-1$.

Contraction comes with a natural multiplicity given by counting the number of hyperplanes of $\A'$ that give the same intersection. A modified version of such a notion will be thoroughly used in this work.

The following notion will play a substantial role in some inductive arguments throughout the paper: $\ell_i$  is called a {\em coloop} if the rank of the deletion $\A'$ with respect to $\ell_i$ is $k-1$, ie. drops by one. This simply means that $\bigcap_{j\neq i} H_j$ is a line rather than the origin in $\mathbb{A}^k$. Otherwise, we say that $\ell_i$ is a {\em non-coloop}.

\subsection{The Orlik-Terao algebra.}\label{OrTe}

One of our motivations is to clarify the connections between the Rees algebra and the Orlik-Terao algebra which is an important object in the theory of hyperplane arrangements. We state the definition and review some of its basic properties below. 

Let $\A\subset\mathbb P^{k-1}$ be a hyperplane arrangement as above. Suppose $c_{i_1}\ell_{i_1}+\cdots+c_{i_m}\ell_{i_m}=0$ is a linear dependency among $m$ of the linear forms defining $\A$,  denoted $D$.  Consider the following homogeneous polynomial  in $S:=\K[y_1,\ldots,y_n]$:
\begin{equation}\label{partial_of_dependency}
\partial D:=\sum_{j=1}^m c_{i_j}\prod_{j\neq k=1}^m y_{i_k}.
\end{equation}
Note that $\deg(\partial D)=m-1$.

The {\em Orlik-Terao algebra of $\A$} is the standard graded $\K$-algebra
\[
\OT(\A):=S/\partial(\A),
\]
where $\partial(\A)$ is the ideal of $S$ generated by $\{\partial D|D\mbox{ a dependency of }\A\}$, with $\partial D$ as in (\ref{partial_of_dependency}) --
called the {\em Orlik-Terao ideal}.
This algebra was introduced in \cite{OrTe} as a commutative analog of the classical combinatorial Orlik-Solomon algebra, in order to answer a question of Aomoto.

The following remark states a few important properties of this algebra.


\begin{rem}\label{OTproperties}\leavevmode
	\begin{enumerate}[i.]
		\item  Recalling that a circuit is a minimally dependent set, one has that $\partial(\A)$ is generated by $\partial C$, where $C$ runs over the circuits of $\A$ (\cite{OrTe2}). In addition, these generators form an universal Gr\"{o}bner basis for $\partial(\A)$ (\cite{PrSp}).
		
		\item  $\OT(\A)$ is Cohen-Macaulay (\cite{PrSp}).
		
		\item  $\OT(\A)\simeq \K[1/\ell_1,\dots, 1/\ell_n]$,  a $k$-dimensional $\K$-subalgebra of the field of fractions $\K(x_1,\ldots,x_k)$ (\cite{ScTo,Te}). The corresponding projective variety is called the {\em reciprocal plane} and it is denoted by $\recip$.
		
		\item  Although the Orlik-Terao algebra is sensitive to the linear forms defining $\A$, its Hilbert series only depends on the underlying combinatorics (\cite{Te}). Let 
		\[
		\pi(\A,s)=\sum_{F\in L(\A)}\mu_{\A}(F)(-s)^{r(F)}.
		\]
		be the {\em \poin polynomial} where $\mu_{\A}$ denotes the M\"{o}bius function, $r$ is the rank function and $F$ runs over the flats of $\A$. Then, we have
		\begin{eqnarray*}
			HS(\OT(\A),s)  =  \pi(\A,\frac{s}{1-s}).
		\end{eqnarray*}
         	See \cite{OrTe} for details and \cite{Te} and \cite{Be} for proofs of the above statement.
		
	\end{enumerate}
\end{rem}

\subsection{Ideals of products from arrangements.}\label{products}

Let $\A=\{\ell_1,\ldots,\ell_n\}$ denote a central arrangement in $R:=\K[x_1,\ldots,x_k]$, $n\geq k$. Denoting $[n]:=\{1,\ldots,n\}$, if $S\subset [n]$, then we set $\ell_S:=\prod_{i\in S}\ell_i, \ell_{\emptyset}:=1$. Also set $S^c:=[n]\setminus S$.

Let $\mathfrak{S}:=\{S_1,\ldots,S_m\}$, where $S_j\subseteq [n]$ are subsets of the same size $e$. We are interested in studying the Rees algebras of ideals of the form
\begin{equation}\label{ideal_of_arbitrary_products}
I_{\mathfrak{S}}:=\langle \ell_{S_1},\ldots,\ell_{S_m}\rangle\subset R.
\end{equation}

\begin{exm}\label{Example_folds}
(i) (The Boolean case) Let $n=k$ and $\ell_i=x_i, i=1,\ldots,k$. Then the ideal $I_{\mathfrak{S}}$ is monomial for any ${\mathfrak{S}}$. In the simplest case where $e=n-1$, it is the ideal of the partial derivatives of the monomial $x_1\cdots x_k$ -- also the base ideal of the classical M\"obius involution. For $e=2$ the ideal becomes the edge ideal of a simple graph with $k$ vertices.
 In general, it gives a subideal of the ideal of paths of a given length on the complete graph and, as such, it has a known combinatorial nature.

 (ii) ($(n-1)$-fold products) Here one takes $S_1:=[n]\setminus\{1\},\ldots,S_n:=[n]\setminus\{n\}.$
 We will designate the corresponding ideal by $I_{n-1}(\A)$.
 This case will be the main concern of the paper and will be fully examined in the following sections.
 
 (iii) ($a$-fold products) This is a natural extension of (ii), where $I_a(\A)$ is the ideal generated by all distinct $a$-products of the linear forms defining $\A$. The commutative algebraic properties of these ideals connect strongly to properties of the linear code built on the defining linear forms (see \cite{AnGaTo}).
 In addition, the dimensions of the vector spaces generated by $a$-fold products give a new interpretation to the Tutte polynomial of the matroid of $\A$ (see \cite{Be}).
 
\end{exm}

We can naturally introduce the {\em reciprocal plane algebra}
\begin{equation}
\mathcal{L}_{\mathfrak S}^{-1}:=\K\left[\frac{1}{\ell_{S_1^c}},\ldots,\frac{1}{\ell_{S_m^c}}\right]
\end{equation}
as a generalized version of the notion mentioned in Remark~\ref*{OTproperties} (iii).

\begin{prop}\label{fiber_general} In the above setup there is a graded isomorphism of $\K$-algebras
	$$\mathcal \K[\ell_{S_1},\ldots,\ell_{S_m}]\simeq\K\left[\frac{1}{\ell_{S_1^c}},\ldots,\frac{1}{\ell_{S_m^c}}\right].$$
\end{prop}
\noindent {\it Proof.} Consider both algebras as homogeneous $\K$-subalgebras of the homogeneous total quotient ring of the standard polynomial ring $R$, generated in degrees $e$ and $-(d-e)$, respectively.
Then multiplication by the total product $\ell_{[d]}$ gives the required isomorphism:
$$\kern5cm\K\left[\frac{1}{\ell_{S_1^c}},\ldots,\frac{1}{\ell_{S_m^c}}\right]\stackrel{\cdot \ell_{[d]} }\longrightarrow \K[\ell_{S_1},\ldots,\ell_{S_m}]. \quad\quad\quad\quad\quad\quad\quad\quad\quad\quad\quad \quad\square$$

\medskip

A neat consequence is the following result:
\begin{thm}\label{ot_specialfiber}
Let $\A$ denote a central arrangement of size $n$, let $\mathfrak{S}:=\{S_1,\ldots,S_m\}$ be a collection of subsets of $[n]$ of the same size and let $I_{\mathfrak{S}}$ be as in {\rm (\ref{ideal_of_arbitrary_products})}.
Then the reciprocal plane algebra $\mathcal{L}_{\mathfrak S}^{-1}$ is isomorphic to the special fiber of the ideal $I_{\mathfrak{S}}$ as graded $\K$-algebras.
In particular, the Orlik-Terao algebra $\OT(\A)$ is graded isomorphic to the special fiber $\fiber$ of the ideal $I=I_{n-1}(\A)$ of $(n-1)$-fold products of $\A$.
\end{thm}
\begin{proof}
It follows immediately from Proposition~\ref{fiber_general} and Lemma~\ref{fiber_principle}.
\end{proof}

\begin{rem}\label{Schenck} In the case of the Orlik-Terao algebra, the above result gives an answer to the third question  at the end of \cite{Sc}. Namely, let $k\geq 3$ and consider the rational map $\Phi$ as in (\ref{map}).
Then Theorem \ref{ot_specialfiber} says that the projection of the graph of $\Phi$ onto the second factor  coincides with the reciprocal plane $\recip$ (see Remark~\ref{OTproperties} (iii)).
In addition, the ideal $I:=I_{n-1}(\A)$ has a similar primary decomposition as obtained in \cite[Lemmas 3.1 and 3.2]{Sc}, for arbitrary $k\geq 3$. By \cite[Proposition 2.2]{AnGaTo}, one gets
	\[
	I = \bigcap_{Y\in L_2(\A)}I(Y)^{\mu_{\A}(Y)},
	\]
\end{rem}

Theorem~\ref{ot_specialfiber} contributes additional information on certain numerical invariants and properties in the strict realm of commutative algebra and algebraic geometry.

\begin{cor} \label{corollary} Let $I:=I_{n-1}(\A)$ denote the ideal generated by the $(n-1)$-fold products coming from a central arrangement of size $n$ and rank $k$.
	One has:
	\begin{itemize}
		\item[(a)] The special fiber $\fiber$ of $I$ is Cohen-Macaulay.
		\item[(b)] The analytic spread is $\ell(I)=k$.
		\item[(c)] The map $\Phi$ is birational onto its image.
		\item[(d)] The reduction number is $r(I)=k-1$.
	\end{itemize}
\end{cor}
\begin{proof} (a) It follows from  from Theorem \ref{ot_specialfiber} via Remark \ref{OTproperties} (ii).
	
	(b) It follows by the same token from \ref{OTproperties} (iii).
	
	(c) This follows from  and \cite[Theorem 3.2]{DoHaSi} since the ideal $I$ is linearly presented (see proof of Lemma \ref{sym_ideal}), and $\ell(I)=k$, maximum possible.
	
	(d) Follows from Part (a), Proposition~\ref{CMfiber}, and \cite[Theorem 3.7]{Sc}.
	
	It may be interesting to remark that, because of this value, in particular the Orlik-Terao algebra is the homogeneous coordinate ring of a variety of minimal degree if and only if $k=2$, in which case it is the homogeneous coordinate ring of the rational normal curve.
\end{proof}

\section{Ideals of $(n-1)$-fold products and their blowup algebras}\label{Basics_on_(n-1)}

As mentioned in Example \ref{Example_folds}, a special case of the ideal $I_{\mathfrak S}$, extending the case of the ideal generated by the $(n-1)$-fold products, is obtained by fixing $a\in\{1,\ldots,n\}$ and considering the collection of all subsets of $[n]$ of cardinality $a$. Then the corresponding ideal is
\[
I_a(\A):=\langle \ell_{i_1}\cdots\ell_{i_a}|1\leq i_1<\cdots<i_a\leq n \rangle\subset R
\]
and is called the ideal generated by the $a$-fold products of linear forms of $\A$. The projective schemes defined by these ideals are known as generalized star configuration schemes. Unfortunately, only  few things are known about these ideals: if $d$ is the minimum distance of the linear code built from the linear forms defining $\A$ and if $1\leq a\leq d$, then $I_a(\A)=\fm^a$ (cf. \cite[Theorem 3.1]{To1}); and the case when $a=n$ is trivial.

In the case where $a=n-1$, some immediate properties are known already, yet the more difficult questions in regard to the blowup and related algebras have not been studied before. These facets, to be throughly examined in the subsequent sections, is our main endeavor in this work.

Henceforth, we will be working with the following  data: $\A$ is an arrangement with $n\geq k$ and for every $1\leq i\leq n$,  we consider the $(n-1)$-fold products of the $n$ linear forms defining the hyperplanes of $\A$
\[
f_i:= \ell_1\cdots \hat{\ell_i}\cdots \ell_n\in R,
\]
and write
\[
I:=I_{n-1}(\A):=\langle f_1,\ldots,f_n\rangle.
\]
Let $T=\K[x_1,\ldots,x_k,y_1,\ldots,y_n]=R[y_1,\ldots,y_n]$ as before and denote by $\mathcal I(\A,n-1)\subset T$ the  presentation ideal of the Rees algebra $R[It]$ corresponding to the generators $f_1,\ldots,f_n$.

\subsection{The symmetric algebra}

Let $\mathcal I_1(\A,n-1)\subset T$ stand for the subideal of $\mathcal I(\A,n-1)$ presenting the symmetric algebra $\mathcal S(I)$ of  $I=I_{n-1}(\A)$.

\begin{lem}\label{sym_ideal} With the above notation, one has:
	\begin{enumerate}
		\item[{\rm (a)}] The ideal $I=I_{n-1}(\A)$ is a perfect ideal of codimension $2$.
		\item[{\rm (b)}] $\mathcal I_1(\A,n-1)=\langle \ell_iy_i-\ell_{i+1}y_{i+1}|1\leq i\leq n-1\rangle $.
		\item[{\rm (c)}] $\mathcal I_1(\A,n-1)$ is an ideal of codimension $k;$ in particular, it is a complete intersection if and only if $n=k$.
	\end{enumerate}
\end{lem}
\begin{proof} (a) This is well-known, but we give the argument for completeness. Clearly, $I$ has codimension $2$. The following reduced Koszul like relations are syzygies of $I$: $\ell_iy_i-\ell_{i+1}y_{i+1}, 1\leq i\leq n-1$.
They alone form the following matrix of syzygies of $I$:
		$$\varphi=\left[
	\begin{array}{rrrr}
	\ell_1&&&\\
	-\ell_2&\ell_2&&\\
	&-\ell_3&\ddots&\\
	&&\ddots&\ell_{n-1}\\
	&&&-\ell_n
	\end{array}
	\right].$$
	
Since the rank of this matrix is $n-1$, it is indeed a full syzygy matrix of $I$; in particular, $I$  has linear resolution
$$0\longrightarrow R(-n)^{n-1}\stackrel{\varphi}{\longrightarrow} R(-(n-1))^n\longrightarrow I\longrightarrow 0.$$
	
(b) This is an expression of the details of (a).

(c)	
Clearly, $\mathcal I_1(\A,n-1)\subset \fm T$, hence its codimension is at most $k$. Assuming, as we may, that $\{\ell_1,\ldots,\ell_k\}$ is $\K$-linearly independent, we contend that the elements $\mathfrak s:=\{\ell_iy_i-\ell_{i+1}y_{i+1}, 1\leq i\leq k\}$, form a regular sequence.
To see this, we first apply a $\K$-linear automorphism of $R$ to assume that $\ell_i=x_i$, for $1\leq i\leq k$ -- this will not affect the basic ideal theoretic invariants associated to $I$.
Then note that in the set of generators of $\mathcal I_1(\A,n-1)$ the elements of $\mathfrak s$ can be replaced by the following ones:
$\{x_iy_i-\ell_{k+1}y_{k+1}, 1\leq i\leq k\}$.
Clearly, this is a regular sequence -- for example, because $\langle x_iy_i, 1\leq i\leq k\rangle$ is the initial ideal of the ideal generated by this sequence, in the revlex order.
\end{proof}

There are two basic ideals that play a distinguished role at the outset.
In order to capture both in one single blow, we consider the Jacobian matrix of the generators of $\mathcal I_1(\A,n-1)$ given in Lemma~\ref{sym_ideal} (b).
Its transpose turns out to be the stack of two matrices, the first is the Jacobian matrix with respect to the variables $y_1,\ldots, y_n$
 -- which coincides with the syzygy matrix $\phi$ of $I$ as described in the proof of Lemma~\ref{sym_ideal} (a) -- while the second is the Jacobian matrix $B=B(\phi)$ with respect to the variables $x_1,\ldots,x_k$ -- the so-called {\em Jacobian dual matrix} of \cite{jac_dual}.
 The offspring are the respective ideals of maximal minors of these stacked matrices, the first retrieves $I$, while the second gives an ideal $I_k(B)\subset S=\K[y_1,\ldots,y_n]$ that will play a significant role below (see also Proposition~\ref{fiber_type_generic}) as a first crude approximation to the Orlik-Terao ideal.

\begin{prop}\label{structure_of_symmetric}
Let $\mathcal S(I)\simeq T/\mathcal I_1(\A,n-1)$ stand for the symmetric algebra of the ideal $I$ of $(n-1)$-fold products. Then:
\begin{enumerate}
	\item[{\rm (i)}] ${\rm depth}(\mathcal S(I))\leq k+1$.
	\item[{\rm (ii)}] As an ideal in $T$, every minimal prime of $\mathcal S(I)$ is either $\fm T$, the Rees ideal $\mathcal I(\A,n-1)$  or else has the form $(\ell_{i_1},\ldots, \ell_{i_s}, y_{j_1},\ldots, y_{j_t})$, where $2\leq s\leq k-1, t\geq 1$, $\{i_1,\ldots,i_s\}\cap\{j_1,\ldots,j_t\}=\emptyset$, and $\ell_{i_1},\ldots, \ell_{i_s}$ are $\K$-linearly independent.
	\item[{\rm (iii)}] The primary components relative to the minimal primes $\fm=(\xx)T$ and $\mathcal I(\A,n-1)$ are radical$;$ in addition, with the exception of $\fm T$, every minimal prime of $\mathcal S(I)$ contains the ideal $I_k(B)$.
\end{enumerate}
\end{prop}
\begin{proof} (i) Since  $\mathcal I(\A,n-1)$ is a prime ideal which is a saturation of $\mathcal I_1(\A,n-1)$ then it is an associated prime of  $\mathcal S(I)$.
Therefore,  ${\rm depth}(\mathcal S(I))\leq \dim \mathcal R(I)=k+1$.

(ii) Since $\mathcal I(\A,n-1)$ is a saturation of $\mathcal I_1(\A,n-1)$ by $I$, one has $\mathcal I(\A,n-1)\, I^t\subset \mathcal I_1(\A,n-1)$, for some $t\geq 1$.
This implies that any (minimal) prime of $\mathcal S(I)$ in $T$ contains either $I$ or $\mathcal I(\A,n-1)$.
By the proof of (i), $\mathcal I(\A,n-1)$ is an associated prime of $\mathcal S(I)$, hence it must be a minimal prime thereof since a minimal prime of $\mathcal S(I)$ properly contained in it would have to contain $I$, which is absurd.

Now, suppose $P\subset T$ is a minimal prime of $\mathcal S(I)$ containing $I$.
One knows by Lemma~\ref{sym_ideal} that $\fm=(\xx)T$ is a minimal prime of $\mathcal S(I)$. Therefore, we assume that $\fm T\not\subset P$.
Since any minimal prime of $I$ is a complete intersection of two distinct linear forms of $\A$ then $P$ contains at least two, and at most $k-1$, linearly independent linear forms of $\A$. On the other hand, since $\mathcal I_1(\A,n-1)\subset P$, looking at the generators of $\mathcal I_1(\A,n-1)$ as in Lemma~\ref{sym_ideal} (b), by a domino effect principle we finally reach the desired format for $P$ as stated.

(iii) With the notation prior to the statement of the proposition, we claim the following equality:
$$ \mathcal I_1(\A,n-1):I_k(B)^{\infty}=\fm T
$$
It suffices to show for the first quotient as $\fm T$ is a prime ideal.
The inclusion $\fm \, I_k(B)\subset \mathcal I_1(\A,n-1)$ is a consequence of the Cramer rule.
The reverse inclusion is obvious because $\mathcal I_1(\A,n-1)\subset \fm T$ implies that $\mathcal I_1(\A,n-1):I_k(B)\subset \fm T:I_k(B)=\fm T$, as $\fm T$ is a prime ideal.

Note that, as a very crude consequence, one has $I_k(B)\subset \mathcal I (\A,n-1)$.

Now, let $\mathcal P(\fm T)$ denote the primary component of $\fm T$ in $\mathcal I_1(\A,n-1)$.
Then
$$\fm T=\mathcal I_1(\A,n-1):I_k(B)^{\infty}\subset \mathcal P(\fm T):I_k(B)^{\infty}=\mathcal P(\fm T).$$
The same argument goes through for the primary component of $\mathcal I (\A,n-1)$ using the ideal $I$ instead of $I_k(B)$.

To see the last statement of the item, let $\mathcal P$ denote the primary component of one of the  remaining minimal primes $P$ of $\mathcal S(I)$.
Since $P:I_k(B)^{\infty}$ is $P$-primary and $\fm\not\subset P$, then by the same token we get that $I_k(B)\subset P$.
\end{proof}

\begin{rem}\label{reducedness}\rm
(a)	It will be shown in the last section that the estimate in (i) is actually an equality.

As a consequence, every associated prime of  $\mathcal S(I)$ viewed in $T$ has codimension at most $n-1$. This will give a much better grip on the minimal primes of the form $\langle \ell_{i_1},\ldots, \ell_{i_s}, y_{j_1},\ldots, y_{j_t}\rangle $. Namely, one must have in addition that $s+t\leq n-1$ and, moreover, due to the domino effect principle, one must have $s=k-1$, hence $t\leq n-k$.

(b) We conjecture that  $\mathcal S(I)$ is reduced.

The property $(R_0)$ of Serre's is easily verified due to the format of the Jacobian matrix as explained before the above proposition.
The problem is, of course, the property $(S_1)$, the known obstruction for the existence of embedded associated primes.
The case where $n=k+1$, is easily determined.
Here the minimal primes are seen to be $\fm$, $\langle x_1,\ldots, x_{k-1}, y_k\rangle $ and the Rees ideal $\langle \mathcal I_1(\A,k), \partial\rangle$, where $\partial$ is the relation corresponding to the unique circuit.
A calculation will show that the three primes intersect in $\mathcal I_1(\A,k)$.
As a side, this fact alone implies that the maximal regular sequence in the proof of Lemma~\ref{sym_ideal} (c) generates a radical ideal.
For $n\geq k+2$ the calculation becomes sort of formidable, but we will prove later on that the Rees ideal is of fiber type.

(c) The weaker question as to whether the minimal component of  $\mathcal S(I)$  is radical seems pliable.
\end{rem}

If the conjectural statement in Remark~\ref{reducedness} (b) is true then, for any linear form $\ell=\ell_i$ the following basic formula holds
$$ \mathcal I_1(\A,n-1):\ell=\mathcal I(\A,n-1)\cap \left(\bigcap_{\ell\notin P} P\right),
$$
where $P$ denotes a minimal prime other that $\fm T$ and $\mathcal I(\A,n-1)$, as described in proposition~\ref{structure_of_symmetric} (i).
Thus one would recover sectors of the Orlik-Terao generators inside this colon ideal.

Fortunately, this latter virtual consequence holds true and has a direct simple proof.
For convenience of later use, we state it explicitly.
Let $\partial(\A|_{\ell})$ denote the subideal of $\partial(\A)$ generated by all polynomial relations $\partial$ corresponding to minimal dependencies (circuits) involving the linear form $\ell\in \A$.

\begin{lem} \label{lemma2} $\partial(\A|_{\ell})\subset \mathcal I_1(\A,n-1):\ell.$
\end{lem}
\begin{proof} Say, $\ell=\ell_1$. Let $D:\, a_1\ell_1+a_2\ell_2+\cdots+a_s\ell_s=0$ be a minimal dependency involving $\ell_1$, for some $3\leq s\leq n$. In particular, $a_i\neq 0,i=1,\ldots,s$. The corresponding generator of $\partial(\A|_{\ell_1})$ is
	$$\partial D:=a_1y_2y_3\cdots y_s+a_2y_1y_3\cdots y_s+\cdots+a_sy_1y_2\cdots y_{s-1}.$$
	
The following calculation is straightforward.
	
	\begin{eqnarray}
	\ell_1\partial D&=&a_1\ell_1y_2y_3\cdots y_s+(\ell_1y_1-\ell_2y_2)(a_2y_3\cdots y_s+\cdots+a_sy_2\cdots y_{s-1})\nonumber\\
	&+&\ell_2y_2(a_2y_3\cdots y_s+\cdots+a_sy_2\cdots y_{s-1})\nonumber\\
	&=&(a_1\ell_1+a_2\ell_2)y_2y_3\cdots y_s+(\ell_1y_1-\ell_2y_2)(a_2y_3\cdots y_s+\cdots+a_sy_2\cdots y_{s-1})\nonumber\\
	&+&\ell_2y_2(a_3y_2y_4\cdots y_s+\cdots+a_sy_2y_3\cdots y_{s-1})\nonumber\\
	&=&(-a_3\ell_3-\cdots-a_s\ell_s)y_2y_3\cdots y_s+(\ell_1y_1-\ell_2y_2)(a_2y_3\cdots y_s+\cdots+a_sy_2\cdots y_{s-1})\nonumber\\
	&+&\ell_2y_2^2(a_3y_4\cdots y_s+\cdots+a_sy_3\cdots y_{s-1})\nonumber\\
	&=&(\ell_1y_1-\ell_2y_2)(a_2y_3\cdots y_s+\cdots+a_sy_2\cdots y_{s-1})+y_2(\ell_2y_2-\ell_3y_3)a_3y_4\cdots y_s\nonumber\\
	&+&\cdots +y_2(\ell_2y_2-\ell_sy_s)a_sy_3\cdots y_{s-1}\nonumber\\
	&=&a_2y_3\cdots y_s(\ell_1y_1-\ell_2y_2)+a_3y_2y_4\cdots y_s(\ell_1y_1-\ell_3y_3)+\cdots+a_sy_2\cdots y_{s-1}(\ell_1y_1-\ell_sy_s).\nonumber
	\end{eqnarray}
	
	Hence the result.
\end{proof}

\subsection{Sylvester forms}

The Orlik-Terao ideal $\partial(\A)$ has an internal structure of classical flavor, in terms of Sylvester forms.

\begin{prop}\label{ot_ideal}  The generators $\partial(\A)$ of the Orlik-Terao ideal are Sylvester forms obtained from the generators of the presentation ideal $\mathcal I_1(\A,n-1)$ of the symmetric algebra of $I$.
\end{prop}
\begin{proof} Let $D$ be a dependency $c_{i_1}\ell_{i_1}+\cdots+c_{i_m}\ell_{i_m}=0$ with all coefficients $c_{i_j}\neq 0$. Let $ f=\prod_{i=1}^n\ell_i$. Evaluating the Orlik-Terao element $\partial D$ on the products we have
\[
\partial D(f_1,\ldots,f_n)=\sum_{j=1}^m c_{i_j}\frac{f^{m-1}}{\Pi_{j\neq k=1}^m \ell_{i_k}}=\sum_{j=1}^m c_{i_j}\frac{f^{m-1}}{\Pi_{k=1}^m \ell_{i_k}} \ell_{i_j}=\frac{f^{m-2}}{\ell_{i_1}\cdots \ell_{i_m}}\left(c_{i_1}\ell_{i_1}+\cdots+c_{i_m}\ell_{i_m}\right)=0.
\]

 Therefore, $\partial D\in \mathcal I(\A,n-1)$, and since $\partial D\in S:=\K[y_1,\ldots,y_n]$, then $\partial D\in \langle \mathcal  I(\A,n-1)_{(0,-)}\rangle$.

\medskip

For the second part, suppose that the minimal generators of $\mathcal I_1(\A,n-1)$ are $$\Delta_1:=\ell_1y_1-\ell_2y_2, \Delta_2:= \ell_2y_2-\ell_3y_3,\ldots,\Delta_{n-1}:=\ell_{n-1}y_{n-1}-\ell_ny_n.$$

Without loss of generality suppose $\ell_j=c_1\ell_1+\cdots+c_{j-1}\ell_{j-1}$ is some arbitrary dependency $D$. We have
$$\left[\begin{array}{l} \Delta_1\\ \Delta_2\\\vdots\\\Delta_{j-1}
\end{array}\right]=\left[
\begin{array}{cccccc}
y_1&-y_2&0&\cdots&0&0\\
0&y_2&-y_3&\cdots&0&0\\
\vdots&\vdots&\vdots& &\vdots &\vdots\\
0&0&0&\cdots&y_{j-2}&-y_{j-1}\\
-c_1y_j&-c_2y_j&-c_3y_j&\cdots&-c_{j-2}y_j&y_{j-1}-c_{j-1}y_j
\end{array}
\right]\cdot \left[\begin{array}{l} \ell_1\\ \ell_2\\\vdots\\\ell_{j-1}
\end{array}\right].$$ The determinant of the $(j-1)\times(j-1)$ matrix we see above is $\pm\partial D$.
\end{proof}

\subsection{A lemma on deletion} \label{del}
In this and the next parts we build on the main tool of an inductive procedure.

Let $\A'=\A\setminus\{\ell_1\}$, and denote $n':=|\A'|=n-1$. We would like to investigate the relationship between the Rees ideal $\mathcal I(\A',n'-1)$ of $I_{n'-1}(\A')$ and the Rees ideal $\mathcal I(\A,n-1)$ of $I_{n-1}(\A)$, both defined in terms of the naturally given generators.

To wit, we will denote the generators of $I_{n'-1}(\A')$ as
\[
f_{12}:= \ell_{[n]\setminus \{1,2\}}, \dots , f_{1n}:= \ell_{[n]\setminus \{1,n\}}.
\]

One can move between the two ideals in a simple manner, which is easy to verify:
 $$I_{n-1}(\A):\ell_1=I_{n'-1}(\A').$$

Note that the presentation ideal $\mathcal I(\A',n'-1)$ of the Rees algebra of $I_{n'-1}(\A')$ with respect to these generators lives in the polynomial subring $T':=R[y_2,\ldots,y_n]\subset T:=R[y_1,y_2,\ldots,y_n]$.
From Lemma~\ref{sym_ideal}, we know that
$$\mathcal I_1(\A',n'-1)T=\langle \ell_2y_2-\ell_3y_3,\ell_3y_3-\ell_4y_4,\ldots,\ell_{n-1}y_{n-1}-\ell_ny_n\rangle T \subset \mathcal I_1(\A,n-1).$$
Likewise, for the Orlik-Terao ideal (which is an ideal in $S':=\K[y_2,\ldots,y_n]\subset S:=\K[y_1,y_2,\ldots,y_n]$), it obtains via Theorem \ref{ot_specialfiber}:
 $$\partial(\A')S=\langle\mathcal I(\A',n'-1)_{(0,-)}\rangle S\subset \partial(\A)=\langle \mathcal  I(\A,n-1)_{(0,-)}\rangle.$$

\begin{lem} \label{lemma1} One has $$\mathcal I(\A,n-1)=\langle \ell_1y_1-\ell_2y_2,\mathcal I(\A',n'-1)\rangle:\ell_1^{\infty}.$$
\end{lem}
\begin{proof}
The inclusion $\langle \ell_1y_1-\ell_2y_2,\mathcal I(\A',n'-1)\rangle:\ell_1^{\infty}\subset  I(\A,n-1)$ is clear since we are saturating a subideal of a prime ideal by an element not belonging to the latter.
We note that the codimension of $\langle \ell_1y_1-\ell_2y_2,\mathcal I(\A',n'-1)\rangle$ exceeds by $1$ that of $I(\A',n'-1)$ since the latter is a prime ideal even after extending to the ambient ring $T$. Therefore, by a codimension counting it would suffice to show that the saturation is itself a prime ideal.

Instead, we choose a direct approach. Thus, let $F\in \mathcal I(\A,n-1)$ be (homogeneous) of degree $d$ in variables $y_1,\ldots,y_n$. We can write $$F=y_1^uG_u+y_1^{u-1}G_{u-1}+ \cdots+y_1G_1+G_0,\, 0\leq u\leq d,$$ where $G_j\in\K[x_1,\ldots,x_k][y_2,\ldots,y_n]$, are homogeneous of degree $d-j$ in $y_2,\ldots,y_n$ for $j=0,\ldots,u$.

Evaluating $y_i=f_i, i=1,\ldots,n$ we obtain

\begin{eqnarray}
0&=&F(f_1,\ldots,f_n)\nonumber\\
&=&\ell_2^uf_{12}^u\ell_1^{d-u}G_u(f_{12},\ldots,f_{1n})+\cdots+\ell_2f_{12}\ell_1^{d-1}G_1(f_{12},\ldots,f_{1n})+\ell_1^dG_0(f_{12},\ldots,f_{1n}). \nonumber
\end{eqnarray}

This means that $$\ell_1^{d-u}\left[\underbrace{\ell_2^uy_2^uG_u(y_2,\ldots,y_n)+\cdots+\ell_1^{u-1}\ell_2y_2G_1(y_2,\ldots,y_n)+\ell_1^uG_0(y_2,\ldots,y_n)}_{F'}\right] \in\mathcal I(\A',n'-1).$$

By writing $\ell_1y_1=\ell_1y_1-\ell_2y_2+\ell_2y_2$, it is not difficult to see that $$\ell_1^uF\equiv F'\mbox{ mod }\langle \ell_1y_1-\ell_2y_2\rangle,$$ hence the result.
\end{proof}


\subsection{Stretched arrangements with coefficients}\label{stretched}
	
Recall the notion of contraction and the inherent idea of a multiarrangement, as mentioned in Section~\ref*{gens}.
Here we wish to consider such multiarrangements, allowing moreover the repeated individual linear functionals corresponding to repeated hyperplanes to be tagged with a nonzero element of the ground field.
For lack of better terminology, we call such a new gadget a {\em stretched arrangement with coefficients}.
Note that, by construction, a stretched arrangement with coefficients $\B$ has a uniquely defined (simple) arrangement $\A$ as support.
Thus, if $\A=\{\ell_1,\dots, \ell_n\}$ is a simple arrangement, then a stretched arrangement with coefficients $\B$ is of the form
\[
\{\underbrace{b_{1,1}\ell_1,\dots,b_{1,m_1}\ell_1}_{H_1=\ker \ell_1},\underbrace{b_{2,1}\ell_2,\dots, b_{2,m_2}\ell_2}_{H_2=\ker \ell_2},\dots, \underbrace{b_{n,1}\ell_n,\dots, b_{n,m_n}\ell_n}_{H_n=\ker \ell_n}\},
\]
where $0\neq b_{i,j}\in \K$ and $H_i=\ker (\ell_i)$ has multiplicity $m_i$ for any $1\leq i\leq n$, and for convenience, we assume that $b_{i,1}=1$. Set $m:=m_1+\cdots+m_n$.
We emphasize the ingredients of a stretched arrangement by writing ${\mathcal B}=(\A, m)$.

Proceeding as in the situation of a simple arrangement, we introduce the collection of $(m-1)$-products of elements of $\B$ and denote $I_{m-1}(\B)$ the ideal of $R$ generated by them.
As in the simple case, we consider the presentation ideal $\mathcal I(\B,m-1)$ of  $I_{m-1}(\B)$ with respect to its set of generators consisting of the $(m-1)$-products. The next lemma relates this ideal to the previously considered presentation ideal $\mathcal I(\A,n-1)$ of $I_{n-1}(\A)$ obtained by taking the set of generators consisting of the $(n-1)$-products of elements of $\A$.

\begin{lem}\label{lem:multiTOsimple}
 Let $\A$ denote an arrangement and let ${\mathcal B}=(\A,m)$ denote a multiarrangement supported on $\A$, as above. Let $G\in R$ stand for the $\gcd$ of the $(m-1)$-products of elements of $\B$. Then:
 \begin{enumerate}
 	\item[{\rm (i)}] The vector of the $(m-1)$-products of elements of $\B$ has the form $G\cdot P_{\A}$, where $P_{\A}$ denotes the vector whose coordinates are the $(n-1)$-products of the corresponding simple $\A$, each such product repeated as many times as the stretching in $\B$ of the corresponding linear form deleted in the expression of the product, and further tagged with certain coefficient
 	\item[{\rm (ii)}] $\mathcal I(\B,m-1)=\langle \mathcal I(\A,n-1), \mathcal D_{\A}\rangle$, where $\mathcal D_{\A}$ denotes the $\K$-linear dependency relations among elements of  $P_{\A}$.	
 \end{enumerate}
\end{lem}
\begin{proof}
	(i) This follows from the definition of a stretched arrangement vis-\`a-vis its support arrangement.
	
	(ii) By (i), the Rees algebra of $I_{m-1}(\B)$ is isomorphic to the Rees algebra of the ideal with generating set $P_{\A}$. By the nature of the latter, the stated result is now clear.
\end{proof}

\section{The main theorems}

We keep the previous notation as in (\ref{del}), where $I_{n-1}(\A)$ is the ideal of $(n-1)$-fold products of a central arrangement $\A$ of size $n$ and rank $k$.
We had $T:=R[y_1,\ldots,y_n]$, with $R:=\K[x_1,\ldots,x_k]$, $S:=\K[y_1,\ldots,y_n]$, and $\mathcal I_1(\A,n-1)\subset \mathcal I(\A,n-1)\subset T$ denote, respectively, the presentation ideals of the symmetric algebra and of the Rees algebra of $I$.
Recall that from Theorem \ref{ot_specialfiber}, the Orlik-Terao ideal  $\partial(\A)$ coincides with the defining ideal $(\mathcal I(\A,n-1)_{(0,-)})S$ of the special fiber algebra of $I$.

\subsection{The case of a generic arrangement.}

Simple conceptual proofs can be given in the case where $\A$ is generic (meaning that any $k$ of the defining linear forms are linearly independent), as follows.

\begin{prop}\label{fiber_type_generic} If $\A=\{\ell_1,\ldots,\ell_n\}\subset R=\K[x_1,\ldots,x_k]$ is a generic arrangement, one has:
	\begin{enumerate}
		\item[{\rm (a)}] $I:=I_{n-1}(\A)$ is an ideal of fiber type.
		\item[{\rm (b)}] The Rees algebra $R[It]$ is Cohen-Macaulay.
		\item[{\rm (c)}] The Orlik-Terao ideal of $\A$ is the $0$-th Fitting ideal of the Jacobian dual matrix of $I$ {\rm (}i.e., the ideal generated by the $k\times k$ minors of the Jacobian matrix of the generators of $\mathcal I_1(\A,n-1)$ with respect to the variables of $R${\rm )}.
        \item[{\rm (d)}] Let $k=n$, ie. the case of Boolean arrangement. Under the standard bigrading $\deg x_i=(1,0)$ and $\deg y_j=(0,1)$, the bigraded Hilbert series of $R[It]$ is
\[
  HS(R[It]; u, v) = \frac{(1-uv)^{k-1}}{(1-u)^k(1-v)^k}.
 \]
	\end{enumerate}
\end{prop}
\begin{proof} 
	As described in the proof of Lemma \ref{sym_ideal}, $I$ is a linearly presented codimension $2$ perfect ideal with syzygy matrix of the following shape
	
	$$\varphi=\left[
	\begin{array}{rrrr}
	\ell_1&&&\\
	-\ell_2&\ell_2&&\\
	&-\ell_3&\ddots&\\
	&&\ddots&\ell_{n-1}\\
	&&&-\ell_n
	\end{array}
	\right].$$
	
	The Boolean case $n=k$ is well-known, so we assume that $\mu(I)=n>k$.
	
	We claim that $I$ satisfies the $G_k$ condition.
	For this purpose we check the requirement in (\ref{G_k in minors}).
	First note that, for $p\geq n-k+1$, one has $$I_p(\varphi)=I_p(\A),$$
	where the rightmost ideal is the ideal generated by all $p$-fold products of the linear forms defining $\A$, as in our earlier notation.
	Because $\A$ is generic, it is the support of the codimension $(n-p+1)$-star configuration $V_{n-p+1}$ (see \cite{GeHaMi}).
By \cite[Proposition 2.9(4)]{GeHaMi}, the defining ideal of $V_{n-p+1}$ is a subset of $I_p(\A)$, hence ${\rm ht}(I_p(\A))\geq n-p+1$. By \cite{To1}, any minimal prime of $I_p(\A)$ can be generated by $n-p+1$ elements. Therefore, ${\rm ht}(I_p(\A))\leq n-p+1$, and hence equality.
	
	
	The three statements now follow from \cite[Theorem 1.3]{MoUr}, where
	(a) and (c) are collected together by saying that $R[It]$ has a presentation ideal of the expected type -- quite stronger than being of fiber type. Note that, as a bonus, \cite[Theorem 1.3]{MoUr} also gives that $\ell(I)=k$ and $r(I)=k-1$, which are parts (b) and (d) in Corollary~\ref{corollary}, when $\A$ is generic.
	
	Part (d) follows from an immediate application of \cite[Theorem 5.11]{RoVa} to the $(k-1)\times k$ matrix
	\[
	 M = \begin{bmatrix} x_1 & 0 & \dots & 0 & x_k \\
	                  0  & x_2 & \dots & 0 & x_k \\
	                  \vdots & &  \ddots & \vdots & x_k \\
	                  0 & 0 & \dots & x_{k-1} & x_k
	 \end{bmatrix}.
	\]

One can verify that the codimension of $I_t(M)$, the ideal of size $t$ minors of $M$, is $k-t+2$. Note that their setup of \cite{RoVa} is different in that they set $\deg y_j=(n-1,1)$, whereas for us $\deg y_j=(0,1)$. To get our formula we make the substitution in their formula: $a\leftrightarrow u$, and $a^{n-1}b \leftrightarrow v$.
\end{proof}

\subsection{The fiber type property.}

In this part we prove one of the main assertions of the section and state a few structural consequences.

\begin{thm}\label{fiber_type}
Let $\A$ be a central arrangement of rank $k\geq 2$ and size $n\geq k$. The ideal  $I_{n-1}(\A)$ of $(n-1)$-fold products of $\A$ is of fiber type:
\[
\mathcal I(\A,n-1)=\langle \mathcal I_1(\A,n-1)\rangle + \langle\mathcal I(\A,n-1)_{(0,-)}\rangle,
\]
as ideals in $T$, where $\langle \mathcal I(\A,n-1)_{(0,-)}\rangle_S =\partial(\A)$ is the Orlik-Terao ideal.

\end{thm}
\begin{proof} We first consider the case where $n=k$. Then $I_{n-1}(\A)$ is an ideal of linear type by Lemma~\ref{sym_ideal}, that is to say, $\mathcal I(\A,n-1)=\mathcal I_1(\A,n-1)$. This proves the statement of the theorem since $\partial(\A)=0$ in this case.
	
We now prove the statement by induction on the pairs $(n,k)$, where $n> k\geq 2$. In the initial induction step, we deal with the case $k=2$ and arbitrary $n> 2$ (the argument will even be valid for $n=2$).
Here one claims that $I_{n-1}(\A)=\langle x_1,x_2\rangle^{n-1}$.
In fact, since no two forms of the arrangement are proportional, the generators of $I_{n-1}(\A)$ are $\K$-linearly independent -- because, e.g., dehomogenizing in one of the variables yields the first $n$ powers of the other variable up to elementary transformations. Also, since these forms have degree $n-1$, they forcefully span the power $\langle x_1,x_2\rangle^{n-1}$.

Now, any $\langle x_1,x_2\rangle$-primary ideal in $\K[x,y]$  automatically satisfies the property $G_2$ (see (\ref{G_k in minors})).
Therefore, the Rees ideal is of fiber type, and in fact it is of the expected type and Cohen-Macaulay by \cite[Theorem 1.3]{MoUr}. In any case, the Rees ideal has long been known in this case, with the defining ideal of the special fiber generated by the $2$-minors of the generic $2\times (n-1)$ Hankel matrix, i.e., by the homogeneous defining ideal of the rational normal curve in $\mathbb P^{n-1}$ (see \cite{Corsini}).

For the main induction step, suppose $n>k>3$ and let $\A':=\A\setminus\{\ell_1\}$ stand for the deletion of $\ell_1$, a subarrangement of size $n':=n-1$. Applying a change of variables in the base ring $R$ -- which, as already remarked, does not disturb the ideal theoretic properties in sight -- we can assume that $\ell_1=x_1$ and $\ell_2=x_2$.
The following extended ideals  $\mathcal I(\A',n'-1)T,\,\partial(\A')S, \,\mathcal I_1(\A',n'-1)T$ will be of our concern.

The following equalities of ideals of $T$ are easily seen to hold:
\begin{eqnarray}\label{inductive_equalities}
\mathcal I_1(\A,n-1) & = & \langle x_1y_1-x_2y_2,\, \mathcal I_1(\A',n'-1)\rangle \quad \mbox{\rm as ideals in} \;T,\\ \nonumber
\partial(\A) & = &\langle \partial(\A|_{x_1}),\, \partial(\A')\rangle, \quad \mbox{\rm as ideals in} \; S.
\end{eqnarray}

Let $F\in \mathcal I(\A,n-1)$ be bihomogeneous with $\deg_{\bf y}(F)=d$.

Suppose that $M=x_1^ay_1^bN \in T$ is a monomial that appears in $F$, where $x_1,y_1\nmid N$. If $a\geq b$, we can write
\[
M=x_1^{a-b}(x_1y_1-x_2y_2+x_2y_2)^bN,
\]
and hence
\[
M\equiv x_1^{a-b}x_2^by_2^bN\mbox{ mod }\langle x_1y_1-x_2y_2\rangle.
\]
If $a<b$, we have $$M=(x_1y_1-x_2y_2+x_2y_2)^ay_1^{b-a}N,$$ and hence $$M\equiv x_2^ay_2^ay_1^{b-a}N\mbox{ mod }\langle x_1y_1-x_2y_2\rangle.$$

Denote $R':=\K[x_2,\dots, x_k]\subset R, T'':=R'[y_2,\dots, y_n]\subset T':=R[y_2,\dots, y_n]\subset T$.

In any case, one can write $$F=(x_1y_1-x_2y_2)Q+x_1^{m_1}P_1+x_1^{m_2}P_2+\cdots+x_1^{m_u}P_u+P_{u+1}, m_1>\cdots>m_u\geq 1,$$
for certain forms $Q\in T$, $P_1,\ldots,P_u\in T''$, and $P_{u+1}\in R'[y_1,\ldots,y_n]=T''[y_1]$ of degree $d$ in the variables $y_1,\ldots,y_n$.

Also
\[
P_{u+1}=y_1^vG_v+y_1^{v-1}G_{v-1}+\cdots+y_1G_1+G_0,
\]
where $G_j\in T''$ and $\deg(G_j)=d-j,j=0,\ldots,v$.


Let us use the following generators (\ref{del}) for $I_{n'-1}(\A')$:
\[
f_{12}:= \ell_{[n]\setminus \{1,2\}}, \dots , f_{1n}:= \ell_{[n]\setminus \{1,n\}}.
\]

Since evaluating $F\in \mathcal I(\A,n-1)$ at $$y_1\mapsto f_1=x_2\ell_3\cdots\ell_n,\,y_2\mapsto f_2=x_1f_{12},\ldots,\, y_n\mapsto f_n=x_1f_{1n}$$ vanishes,
upon pulling out the appropriate powers of $x_1$, it yields
\begin{eqnarray}
0&=&x_1^{m_1+d}P_1(f_{12},\ldots,f_{1n})+\cdots+x_1^{m_u+d}P_u(f_{12},\ldots,f_{1n})\nonumber\\
&+&f_1^vx_1^{d-v}G_v(f_{12},\ldots,f_{1n})+\cdots+f_1x_1^{d-1}G_1(f_{12},\ldots,f_{1n})+x_1^dG_0(f_{12},\ldots,f_{1n}).\nonumber
\end{eqnarray}

Suppose first that the rank of $\A'$ is $k-1$, i.e. $x_1$ is a coloop. This means that $x_2=\ell_2,\ell_3,\ldots,\ell_n$ are actually forms in the subring $R'=\K[x_2,\ldots.x_k]$. Since $m_1+d>\cdots>m_u+d>d>\cdots>d-v$, $$P_i(f_{12},\ldots,f_{1n})=0, i=1,\ldots,u,\quad G_j(f_{12},\ldots,f_{1n})=0,j=0,\ldots,v.$$
Therefore, $P_i,G_j\in\mathcal I(\A',n'-1)$, and hence $F\in \langle x_1y_1-x_2y_2,\mathcal I(\A',n'-1)\rangle$.
This shows that
$$\mathcal I(\A,n-1)=\langle x_1y_1-x_2y_2,\mathcal I(\A',n'-1)\rangle$$
and the required result follows by the inductive hypothesis as applied to $\mathcal I(\A',n'-1)$.

Suppose now that the rank of $\A'$ does not drop, i.e. $x_1$ is a non-coloop.

\smallskip

{\bf Case 1.} $v=0$. In this case, after canceling $x_1^d$, we obtain $$0=x_1^{m_1}P_1(f_{12},\ldots,f_{1n})+\cdots+x_1^{m_u}P_u(f_{12},\ldots,f_{1n})+G_0(f_{12},\ldots,f_{1n}).$$
Thus,
$$x_1^{m_1}P_1+x_1^{m_2}P_2+\cdots+x_1^{m_u}P_u+P_{u+1}\in \mathcal I(\A',n'-1).$$

{\bf Case 2.} $v\geq 1$. In this case we cancel the factor $x_1^{d-v}$  in the above equation. This will give $$x_1|G_v(f_{12},\ldots,f_{1n}).$$

At this point we resort to the idea of stretched arrangements with coefficients as developed in Subsection~\ref{stretched}.
Namely, we take the restriction (contraction) of $\A$ to the hyperplane $x_1=0$.
Precisely, say
$$\ell_i=a_ix_1+\bar{\ell}_i,\mbox{ where }\bar{\ell_i}\in R',a_i \in \K,$$
for $i=2,\ldots,n$.
Note that $a_2=0$ since $\ell_2=x_2$.

Write
\[
\bar{\mathcal A}=\{\bar{\ell}_2,\ldots,\bar{\ell}_n\}\subset R'.  
\]
a stretched arrangement of total multiplicity $\bar{n}=n-1$ with support $\A''$ of size $n''\leq \bar{n}$.

Likewise, let
\[
\bar{f}_{12}:=\bar{\ell}_3\cdots\bar{\ell}_n,\ldots,\bar{f}_{1n}:=\bar{\ell}_2\cdots\bar{\ell}_{n-1}
\]
denote the $(\bar{n}-1)$-products of this stretched arrangement.
Then, $G_v$ vanishes on the tuple $(\bar{f}_{12},\ldots,\bar{f}_{1n})$ and since its is homogeneous it necessarily belong to $\mathcal I(\bar{\A},\bar{n}-1).$

From Lemma \ref{lem:multiTOsimple}, we have
\[
\cI(\bar{\A},\bar{n}-1)=\langle \cI({\A}'',n''-1),\,\mathcal D_{\bar{\A}}\rangle,
\]
where $\mathcal D_{\bar{\A}}$ is a linear ideal of the form $\langle y_i-b_{i,j} y_j\rangle_{2\leq i,j\leq n}$.

\smallskip

Let us analyze the generators of $\mathcal I(\bar{\A},\bar{n}-1)$.

\smallskip

$\bullet$ A generator $y_i-b_{i,j} y_j, i,j\geq 2$ of $\mathcal D_{\bar{\A}}$ comes from the relation $\bar{\ell}_j=b_{i,j}\bar{\ell}_i,b_{i,j}\in\K$. Thus, back in $\A$ we have the minimal dependency
\[
\ell_j-a_jx_1=b_{i,j}(\ell_i-a_ix_1),
\]
yielding an element of $\partial(\A|_{\ell_1})$:

\[
y_1(y_i-b_{i,j}y_j)+\underbrace{(b_{i,j}a_i-a_j)}_{c_{i,j}}y_iy_j.
\]

$\bullet$ Since $\gcd(\bar{\ell}_i,\bar{\ell}_j)=1$, for $2\leq i<j\leq n''+1$, a typical generator of $\mathcal I_1(\A'',n''-1)$ is $\bar{\ell}_iy_i-\bar{\ell}_jy_j$, that we will rewrite as $$\bar{\ell}_iy_i-\bar{\ell}_jy_j=(\ell_iy_i-\ell_jy_j)-x_1(a_iy_i-a_jy_j).$$

$\bullet$ A typical generator of $\partial(\A'')$ is of the form $b_1y_{i_2}\cdots y_{i_s}+\cdots+b_sy_{i_1}\cdots y_{i_{s-1}}$ coming from a minimal dependency
\[
b_1\bar{\ell}_{i_1}+\cdots+b_s\bar{\ell}_{i_s}=0, i_j\in\{2,\ldots,n''+1\}.
\]
Since $\bar{\ell}_{i_j}=\ell_{i_j}-a_{i_j}x_1$, we obtain a dependency
\[
b_1\ell_{i_1}+\cdots+b_s\ell_{i_s}-(\underbrace{b_1a_{i_1}+\cdots+b_sa_{i_s}}_{\alpha})x_1=0.
\]
 If $\alpha=0$, then
\[
b_1y_{i_2}\cdots y_{i_s}+\cdots+b_sy_{i_1}\cdots y_{i_{s-1}}\in\partial(\A'),
\]
whereas if $\alpha\neq0$, then
\[
 -\alpha y_{i_2}\cdots y_{i_s} +y_1(b_1y_{i_2}\cdots y_{i_s}+\cdots+b_sy_{i_1}\cdots y_{i_{s-1}})\in\partial(\A|_{\ell_1}).
\]

\medskip

We have that $$G_v=\sum E_{s,t}(y_s-b_{s,t}y_t)+\sum A_{i,j}(\bar{\ell}_iy_i-\bar{\ell}_jy_j)+\sum B_{i_1,\ldots,i_s}(b_1y_{i_2}\cdots y_{i_s}+\cdots+b_sy_{i_1}\cdots y_{i_{s-1}}),$$ where $E_{s,t},A_{i,j},B_{i_1,\ldots,i_s}\in T''$ and $s,t,i,j,i_k\geq 2$. Then, by using the expressions in the three bullets above and  splicing according to the equality $x_1y_1=(x_1y_1-x_2y_2)+x_2y_2$, we get:

\begin{eqnarray*}
y_1^vG_v&=&y_1^{v-1}(\underbrace{\sum E_{s,t}(y_1(y_s-b_{s,t}y_t)+c_{s,t}y_sy_t)}_{\in\partial(\A|_{\ell_1})}-\underbrace{\sum E_{s,t}c_{s,t}y_sy_t}_{\in T''}\\
&+&\underbrace{\sum A_{i,j}y_1(\ell_iy_i-\ell_jy_j)}_{\in \mathcal I(\A',n'-1)}-\underbrace{\sum A_{i,j}(x_1y_1-x_2y_2)(a_iy_i-a_jy_j)}_{\in \langle x_1y_1-x_2y_2\rangle}-\underbrace{\sum A_{i,j}x_2y_2(a_iy_i-a_jy_j)}_{\in T''}\\
&+&\underbrace{\sum B_{i_1,\ldots,i_s}(y_1(b_1y_{i_2}\cdots y_{i_s}+\cdots+b_sy_{i_1}\cdots y_{i_{s-1}})-\alpha y_{i_2}\cdots y_{i_s})}_{\in \partial(\A|_{\ell_1})}+ \underbrace{\sum B_{i_1,\ldots,i_s}\alpha y_{i_2}\cdots y_{i_s}}_{\in T''}).
\end{eqnarray*}

Thus, $y_1^vG_v=y_1^{v-1}G'_{v-1}+W$, where
\[
G'_{v-1}\in T'',\quad  W\in\langle x_1y_1-x_2y_2,\,\partial(\A|_{\ell_1}),\,\mathcal I(\A',n'-1)\rangle.
\]

Then returning to our original $F$, it obtains $$F=\Delta+x_1^{m_1}P_1+\cdots+x_1^{m_u}P_u+y_1^{v-1}(\underbrace{G'_{v-1}+G_{v-1}}_{G''_{v-1}\in S''})+ y_1^{v-2}G_{v-2}+\cdots+G_0,$$ where $\Delta\in \langle x_1y_1-x_2y_2,\partial(\A|_{\ell_1}),\mathcal I(\A',n'-1)\rangle\subset \mathcal I(\A,n-1)$.

The key  is that modulo the ideal $\langle x_1y_1-x_2y_2,\partial(\A|_{\ell_1}),\mathcal I(\A',n'-1)\rangle$ the power of $y_1$ dropped from $v$ to $v-1$ in the expression of $F$. Iterating, with $F(f_1,\ldots,f_n)=0=\Delta(f_1,\ldots,f_n)$, will eventually drop further the power of $y_1$ to $v-2$. Recursively we end up with $v=0$, which is Case 1 above.

This way, we eventually get
\[
\mathcal I(\A,n-1)=\langle x_1y_1-x_2y_2,\partial(\A|_{\ell_1}),\mathcal I(\A',n'-1)\rangle.
\]

By the inductive hypothesis as applied to $\mathcal I(\A',n'-1)$ and from the two equalities in (\ref{inductive_equalities}), one gets the stated result.
\end{proof}

\begin{cor}\label{corollary1} $\mathcal I(\A',n'-1)=\mathcal I(\A,n-1)\cap T'$
as ideals in $T'=R[y_2,\ldots,y_n]$.
\end{cor}
\begin{proof} Recall the notation $T':=R[y_2,\ldots,y_n]\subset T:=R[y_1,\ldots,y_n]=T'[y_1]$ as in the proof of the previous theorem. Denote $J:=\mathcal I(\A,n-1)\cap T'$. We show that $J\subseteq \mathcal I(\A',n'-1)$, the other inclusion being obvious.

Let $F\in J$. Then $F\in T'$ and $F\in \mathcal I(\A,n-1)$. By Theorem \ref{fiber_type}, we can write $$F=(\ell_1y_1-\ell_2y_2)P+Q+G, \mbox{ where }P\in T,\, Q\in \partial(\A|_{\ell_1})T,\, G\in \mathcal I(\A',n'-1)T.$$

By Lemma~\ref{lemma2},
$$\ell_1Q\in \mathcal I_1(\A,n-1)=\langle \ell_1y_1-\ell_2y_2 , \mathcal I_1(\A',n'-1)\rangle.$$
Therefore,
\begin{equation}\label{key_eq}
\ell_1F=(\ell_1y_1-\ell_2y_2)P'+G',
\end{equation}
for suitable $P'\in T,\, G'\in \mathcal I(\A',n'-1)T$.

We write $P'=y_1^uP_u+\cdots+y_1P_1+P_0, P_i\in T'$, and $G'=y_1^vG_v+\cdots+y_1G_1+G_0, G_j\in T'$.

Since $G'\in \mathcal I(\A',n'-1)\subset T'$,  setting $y_1=0$ in the expression of $G'$ gives that $G_0\in \mathcal I(\A',n'-1)$. Therefore, $G-G_0=y_1(y_1^{v-1}G_v+\cdots+G_1)\in\mathcal I(\A',n'-1),$ and hence $y_1^{v-1}G_v+\cdots+G_1\in \mathcal I(\A',n'-1)$ since $\mathcal I(\A',n'-1)$ is prime. Setting again $y_1=0$ in this expression we obtain that $G_1\in \mathcal I(\A',n'-1)$, and so on, eventually obtaining
$$G_j\in \mathcal I(\A',n'-1), j=0,\ldots,v.$$

Suppose $u\geq v$. Then, by grouping the powers of $y_1$ we have
\begin{eqnarray*}
\ell_1F&=&(-\ell_2y_2P_0+G_0)+(\ell_1P_0-\ell_2y_2P_1+G_1)y_1+\cdots+(\ell_1P_{v-1}-\ell_2y_2P_v+G_v)y_1^v\\
    &+&(\ell_1P_{v}-\ell_2y_2P_{v+1})y_1^{v+1}+\cdots+(\ell_1P_{u-1}-\ell_2y_2P_{u})y_1^{u}+\ell_1P_uy_1^{u+1}.
\end{eqnarray*}

Since $F\in T'$, then $\ell_1F\in T'$. Thus, the ``coefficients'' of $y_1, y_1^2,\ldots,y_1^{u+1}$ must vanish. It follows that
$$P_u=\cdots=P_v=0 \;\; \mbox{\rm and}\;\;\ell_1P_{v-1}=-G_v\in\mathcal I(\A',n'-1).$$ Since $\mathcal I(\A',n'-1)$ is a prime ideal, we have $P_{v-1}\in \mathcal I(\A',n'-1),$ and therefore $$\ell_1P_{v-2}=\ell_2y_2P_{v-1}-G_{v-1}\in\mathcal I(\A',n'-1).$$ Recursively we get that $$P_{v-1},P_{v-2},\ldots,P_1,P_0\in \mathcal I(\A',n'-1).$$

If $u<v$, a similar analysis will give the same conclusion that $P'\in \mathcal I(\A',n'-1)T$.

Therefore, (\ref{key_eq}) gives $\ell_1F\in \mathcal I(\A',n'-1)T,$ and hence $F\in\mathcal I(\A',n'-1)T$  by primality of the extended ideal.
But then $F\in I(\A',n'-1)T\cap T'=I(\A',n'-1)$, as required.
\end{proof}

\medskip

The next two corollaries help compute the Rees ideal from the symmetric ideal via a simple colon of ideals.

\begin{cor}\label{Rees_colon} Let $\ell_i\in\A$ and $y_i$ be the corresponding external variable. Then
\[
\mathcal I(\A,n-1)=\mathcal I_1(\A,n-1):\ell_iy_i.
\]
\end{cor}
\begin{proof} Without loss of generality, assume $i=1$.

The inclusion $\supseteq$ is immediate, since $\mathcal I_1(\A,n-1)\subset\mathcal I(\A,n-1)$, and the Rees ideal $\mathcal I(\A,n-1)$ is a prime ideal not containing $\ell_1$ nor $y_1$.

Now we show the inclusion $\subseteq$. Let $F\in \mathcal I(\A,n-1)$. Then, from Theorem \ref{fiber_type},
\[
F=G+\sum_{D}P_D\partial D,
\] where the sum is taken over all minimal dependencies $D$, and $G\in\mathcal I_1(\A,n-1)$.

Obviously, $\ell_1y_1G\in\mathcal I_1(\A,n-1)$. Also, if $\partial D\in\partial(\A|_{\ell_1})$, then, from Lemma~\ref{lemma2}, $\ell_1\partial D$, hence $\ell_1y_1\partial D$ belongs to $\mathcal I_1(\A,n-1)$.

Suppose $\partial D\notin \partial(\A|_{\ell_1})$. Since $D$ is a minimal dependency among the hyperplanes of $\A$, there exists $j\in\{1,\ldots,n\}$ such that $\partial D\in\partial(\A|_{\ell_j})$. Thus,
$\ell_1y_1\partial D=(\ell_1y_1-\ell_jy_j)\partial D+\ell_jy_j\partial D$ 
belongs to the ideal $\mathcal I_1(\A,n-1)$ since each summand belongs to $\mathcal I_1(\A,n-1)$ -- the first trivially and the second due to Lemma~\ref{lemma2}.
\end{proof}

Since the rank of $\A$ is $k$, after a reordering of the linear forms $\ell_1,\ldots,\ell_n$ that define $\A$, we can assume that the last $k$ linear forms $\ell_{n-k+1},\ldots,\ell_n$ are linearly independent.
With this proviso, one has:

\begin{cor}\label{saturating_by_product} $\mathcal I(\A,n-1)=\mathcal I_1(\A,n-1):\prod_{i=1}^{n-k}\ell_i.$
\end{cor}
\begin{proof} Since $\ell_{n-k+1},\ldots,\ell_n$ are $k$ linearly independent linear forms, any minimal dependency that involves at least one of them, must involve also a linear form $\ell_j$, where $j\in\{1,\ldots,n-k\}$. So
\[
\partial(\A)=\partial(\A|_{\ell_1})+\cdots+\partial(\A|_{\ell_{n-k}}).
\]

We obviously have $ \mathcal I_1(\A,n-1)\subseteq \mathcal I_1(\A,n-1):\prod_{i=1}^{n-k}\ell_i$, and from Lemma~\ref{lemma2},
\[
\partial(\A|_{\ell_j})\subset \mathcal I_1(\A,n-1):\prod_{i=1}^{n-k}\ell_i, \mbox{ for all }j=1,\ldots,n-k.
\]
Then, from Theorem \ref{fiber_type}, one has
\[
\mathcal I(\A,n-1)\subseteq\mathcal I_1(\A,n-1):\prod_{i=1}^{n-k}\ell_i.
\]

The reverse inclusion comes from the fact that $\mathcal I_1(\A,n-1)\subseteq \mathcal I(\A,n-1)$, and from $\mathcal I(\A,n-1)$ being a prime ideal with $\ell_i\notin \mathcal I(\A,n-1)$.
\end{proof}

In the next statement we denote the extended ideal  $(\mathcal I_1(\A',n'-1))T$ by $\langle \mathcal I_1(\A',n'-1)\rangle$.

\begin{lem}\label{lem:x1y1-x2y2}
	Let $\A'=\A\setminus \{\ell_1\}$ and $n'=|\A'|=n-1$. We have
	\[
	\langle \cI_1(\A',n'-1)\rangle :(\ell_1y_1-\ell_2y_2)=\langle \cI_1(\A',n'-1)\rangle:\ell_1.
	\]
In particular, when $\ell_1$ is a coloop, the biform $\ell_1y_1-\ell_2y_2$ is a nonzerodivisor on $\langle \cI_1(\A',n'-1)\rangle$.
\end{lem}
\begin{proof}
	
	For convenience, let us change coordinates to have $\ell_1=x_1$ and $\ell_2=x_2$. Let $f\in \langle \cI_1(\A',n'-1)\rangle:(x_1y_1-x_2y_2)$. Then $f(x_1y_1-x_2y_2)\in \langle \cI_1(\A',n'-1)\rangle\subset \langle \cI(\A',n'-1)\rangle$. Since $\langle \cI(\A',n'-1)\rangle$ is a prime ideal not containing $x_1y_1-x_2y_2$, we obtain $f\in \langle \mathcal I(\A',n'-1)\rangle$, and by Theorem \ref{fiber_type}, we have
	\[
	f=g+h,\quad g\in \langle \cI_1(\A',n'-1)\rangle,\quad h\in \langle\partial(\A')\rangle.
	\]
	By multiplying this by $x_1y_1-x_2y_2$, we get that
	\[
	(x_1y_1-x_2y_2)h\in \langle \cI_1(\A',n'-1)\rangle .
	\]
	By Corollary \ref{Rees_colon}, since $h\in \langle \partial(\A')\rangle \subset \langle \cI(\A',n'-1)\rangle $, and $x_2\in \A'$, we have $x_2y_2h\in \langle \cI_1(\A',n'-1)\rangle $. So $h\in \langle \cI_1(\A',n'-1)\rangle :x_1y_1$, and together with $f=g+h$ with $g\in \langle \cI_1(\A',n'-1)\rangle \subset \mathcal \langle \cI_1(\A',n'-1)\rangle :x_1y_1$, gives
	\[
	f\in \langle \cI_1(\A',n'-1)\rangle:x_1y_1.
	\]
	
	Conversely, let $\Delta\in \langle \cI_1(\A',n'-1)\rangle:x_1y_1$. Then $x_1y_1\Delta\in \langle \cI_1(\A',n'-1)\rangle \subseteq \langle \cI(\A',n'-1)\rangle$. The ideal $\langle \cI(\A',n'-1)\rangle$ is a prime ideal, and $x_1y_1\notin \langle \cI(\A',n'-1)\rangle$, so $\Delta\in \langle \cI(\A',n'-1)\rangle$. So, by Corollary \ref{Rees_colon}, $x_2y_2\Delta\in \langle \cI_1(\A',n'-1)\rangle$. Therefore
	\[
	(x_1y_1-x_2y_2)\Delta=x_1y_1\Delta-x_2y_2\Delta\in \langle \cI_1(\A',n'-1)\rangle.
	\]
	Thus far, we have shown that $	\langle \cI_1(\A',n'-1)\rangle :(\ell_1y_1-\ell_2y_2)=\langle \cI_1(\A',n'-1)\rangle:x_1y_1$.
	Clearly, the right hand side is the same as $\langle \cI_1(\A',n'-1)\rangle:x_1$ since $y_1$ is a nonzero divisor on $\langle \cI_1(\A',n'-1)\rangle$.
\end{proof}

\subsection{The Cohen-Macaulay property}

In this part the goal is to prove that the Rees algebra is Cohen-Macaulay.
Since we are in a graded setting, this is equivalent to showing that its depth with respect to the maximal graded ideal $\langle \fm,y_1,\ldots,y_n\rangle $ is (at least) $k+1=\dim R[It]$.

This will be accomplished by looking at a suitable short exact sequence, where two of the modules will be examined next.
We state the results in terms of depth since this notion is inherent to the Cohen-Macaulay property, yet the proofs will take the approach via projective (i.e., homological) dimension.
By the Auslander-Buchsbaum equality, we are home anyway.

Throughout, ${\rm pdim}_T$ will denote projective dimension over the polynomial ring $T$. Since we are in a graded situation, this is the same as the projective dimension over the local ring $T_{(\fm,y_1,\ldots,y_n)}$, so we may harmlessly proceed.

The first module is obtained by cutting the binomial generators of $\mathcal I_1(\A,n-1)$ into its individual terms.
The result  may have interest on itself.

\begin{lem}\label{conjecture} Let $\ell_1,\ldots,\ell_n\in \K[x_1,\ldots,x_k]$ be linear forms, allowing some of them to be mutually proportional. Let $T:=\K[x_1,\ldots,x_k;y_1,\ldots,y_n]$. Then
	\[
	{\rm depth}\left(\frac{T}{\langle \ell_1y_1,\ell_2y_2,\ldots,\ell_ny_n\rangle}\right)\geq k.
	\]
\end{lem}
\begin{proof} If $k=1$, the claim is clearly satisfied, since $\langle x_1y_1,\ldots,x_1y_n\rangle=x_1\langle y_1,\ldots,y_n\rangle$, and $\{y_1,\ldots,y_n\}$ is a $T$-regular sequence. Assume $k\geq 2$.

We will use induction on $n\geq 1$ to show that the projective dimension is at most $n+k-k=n$.

If $n=1$, the ideal $\langle \ell_1 y_1\rangle$ is a principal ideal, hence the claim is true.	
Suppose $n>1$.
We may apply a $\K$-linear automorphism on the ground variables, which will not disturb the the projective dimension.
Thus, say,  $\ell_1=x_1$ and this form is repeated $s$ times.
Since nonzero coefficients from $\K$ tagged to $x_1$ will not change the ideal in question, we assume that $\ell_i=x_1$, for  $1\leq i\leq s$, and $\gcd(x_1,\ell_j)=1$, for $s+1\leq j\leq n$.
Write $\ell_j=c_jx_1+\bar{\ell}_j$, for $s+1\leq j\leq n$, with $c_j\in\K$, and $0\neq \bar{\ell}_j\in\K[x_2,\ldots,x_k]$.
		
Denoting $J:=\langle x_1y_1,\ldots, x_1y_s,\ell_{s+1}y_{s+1},\ldots,\ell_ny_n\rangle$,
we claim that
\begin{equation}\label{colon_calculation}
J:x_1=\langle y_1,\ldots,y_s, \ell_{s+1}y_{s+1},\ldots,\ell_ny_n\rangle.
\end{equation}

This is certainly the expression of a more general result, but we give a direct proof here.	

One inclusion is obvious. For the reverse inclusion, let $F\in \langle x_1y_1,\ldots, x_1y_s,\ell_{s+1}y_{s+1},\ldots,\ell_ny_n\rangle:x_1$. Then, say,
	\[
	x_1F=x_1\,\sum_{i=1}^s  P_i y_i + \sum_{j=s+1}^n P_j \ell_j y_j,
	\]
for certain $P_i, P_j\in T$.
	Rearranging we have
	\begin{equation}\label{colon_eq}
x_1(F-\sum_{i=1}^s  P_i y_i -  \sum_{j=s+1}^n c_j P_j y_j)=\sum_{j=s+1}^n P_j \bar{\ell_j} y_j.
	\end{equation}
	
	Since $x_1$ is a nonzero divisor in $T/\langle \bar{\ell}_{s+1}y_{s+1},\ldots,\bar{\ell}_ny_n\rangle$, the second factor of the left hand side in (\ref{colon_eq}) must be of the form
	$$\sum_{j=s+1}^n Q_j \bar{\ell_j} y_j, \; Q_j\in T.$$
	Substituting in (\ref{colon_eq}) we find $P_{j}=x_1Q_{j}, s+1\leq j\leq n$, and hence $F=\sum_{i=1}^s P_iy_i +\sum_{j=s+1}^n Q_j \ell_j y_j$, as claimed.
	
	Computing projective dimensions with respect to $T$ and $T'=\K[x_1,\ldots,y_{s+1},\ldots,y_n]$ and applying the inductive hypothesis, one has
	\[
	{\rm pdim}_T\left(\frac{T}{J:x_1}\right)=s+{\rm pdim}_{T'}\left(\frac{T'}{\langle \ell_{s+1}y_{s+1},\ldots,\ell_ny_n\rangle}\right)\leq s+(n-s)=n.
	\]

	At the other end, we have $\langle x_1,J\rangle=\langle x_1,\bar{\ell}_{s+1}y_{s+1},\ldots,\bar{\ell}_ny_n\rangle.$
	Applying the inductive hypothesis this time around gives
	
	\[
	{\rm pdim}_T\left(\frac{T}{\langle x_1,J\rangle}\right)\leq 1+(n-s)\leq n.
	\]
	
	From the short exact sequence of $T$-modules
	\[
	0\longrightarrow T/(J:x_1)\stackrel{\cdot x_1}\longrightarrow T/J\longrightarrow T/\langle x_1,J\rangle\longrightarrow 0,
	\]
	knowingly the projective dimension of the middle term does not exceed the maximum of the projective dimensions of the two extreme terms.
	Therefore, ${\rm pdim}_T(T/J)\leq n,$ as was to be shown..
\end{proof}

\medskip

The difficult result of this section is the following exact invariant of the symmetric algebra $\mathcal S(I)\simeq T/\langle \mathcal I_1(\A,n-1)\rangle$:

\begin{prop} \label{proj_symmetric} Let $I=I_{n-1}(\A)$ as before.
Then ${\rm depth}(\mathcal S(I))=k+1$.
\end{prop}
\begin{proof} By Proposition~\ref{structure_of_symmetric} (i), it suffices to prove the lower bound ${\rm depth}(\mathcal S(I))\geq k+1$.

As in the proof of Theorem~\ref{fiber_type}, we argue by induction on all pairs $n,k$, with $n\geq k\geq 2$, where $n$ and $k$ are, respectively, the size and the rank of $\A$.
	
If $k=2$ and $n>2$, let $R=\K[x,y]$. As seen in that proof, one has $I=\langle x,y\rangle ^{n-1}$, and hence
$$\mathcal I_1(\A,n-1)=\langle xy_1-yy_2, xy_2-yy_3,\ldots, xy_{n-1}-yy_n\rangle.$$
A direct calculation shows that $\{y_1,x+y_n,y+y_{n-1}\}$ is a regular sequence modulo $\mathcal I_1(\A,n-1)$.
	
If $n=k$, the ideal $\mathcal I_1(\A,n-1)$ is a complete intersection by Lemma~\ref{sym_ideal}.
	
Thus, for the main inductive step suppose $n>k>3$.

We will	equivalently show that ${\rm pdim}_T(\mathcal S(I))\leq n-1$.
First apply a change of ground variables so as to have $\ell_1=x_1$ and $\ell_2=x_2$.

Let $\A':=\A\setminus\{x_1\}$ denote deletion.
Since $\mathcal I_1(\A,n-1)=\langle \mathcal I_1(\A',n'-1),x_1y_1-x_2y_2\rangle$, we have the following short exact sequence of $T$-modules
\begin{equation}\label{first_exact_seq}
0\rightarrow\frac{T}{\langle \cI_1(\A',n'-1)\rangle :(x_1y_1-x_2y_2)}\xrightarrow{\cdot (x_1y_1-x_2y_2)} \frac{T}{\langle\cI_1(\A',n'-1)\rangle}\longrightarrow\frac{T}{\mathcal I_1(\A,n-1)}\rightarrow 0.
\end{equation}

We consider separately the cases where $\ell_1$ is a coloop or a non-coloop.

\smallskip

	{\sc $x_1$ is a coloop.}
	
Here the rank of $\A'$ is $k-1$ and $x_1$ is altogether absent in the linear forms of the deletion.
Thus, the natural ambient ring of $\cI_1(\A',n'-1)$ is $T':=\K[x_2,\dots, x_k;y_2,\dots, y_n]$.
	
In this case, by Lemma \ref{lem:x1y1-x2y2}, the left most nonzero term of (\ref{first_exact_seq}) becomes
$$T/\langle\cI_1(\A',n'-1)\rangle=\frac{T'}{\cI_1(\A',n'-1)}[x_1,y_1],$$
hence
$${\rm pdim}_T (T/\langle\cI_1(\A',n'-1)\rangle)={\rm pdim}_{T'}(T'/\cI_1(\A',n'-1))\leq n'-1,$$
by the inductive hypothesis applied to $\mathcal S(I_{n'-1}(\A'))\simeq T'/\cI_1(\A',n'-1)$.

 Then, from (\ref{first_exact_seq}) we have
 \begin{eqnarray*}
{\rm pdim}_T(T/\cI_1(\A,n-1))&\leq &\max\{{\rm pdim}_T(T/\langle\cI_1(\A',n'-1)\rangle)+1,{\rm pdim}_T (T/\langle\cI_1(\A',n'-1)\rangle)\} \\
&\leq & (n'-1)+1=n'=n-1.
 \end{eqnarray*}

	\medskip
	
	{\sc $x_1$ is a non-coloop.}

This case will occupy us for the rest of the proof.	
Here $T':=\K[x_1,\dots, x_k;y_2,\dots,y_n]$ is the natural ambient ring of the deletion symmetric ideal.
	
By Lemma \ref{lem:x1y1-x2y2}, the left most nonzero term of (\ref{first_exact_seq}) is $T/(\langle\cI_1(\A',n'-1)\rangle:x_1)$.
Thus, multiplication by $x_1$ gives a similar exact sequence to (\ref{first_exact_seq}):	
\begin{equation}\label{second_exact_seq}
	0\longrightarrow\frac{T}{\langle \cI_1(\A',n'-1)\rangle:x_1}\longrightarrow \frac{T}{\langle \cI_1(\A',n'-1)\rangle}\longrightarrow\frac{T}{\langle x_1,\mathcal I_1(\A',n'-1)\rangle}\longrightarrow 0.
\end{equation}
Suppose for a minute that one has
\begin{equation}\label{auxiliary_pdim}
{\rm pdim}_T\left(\frac{T}{\langle x_1,\mathcal I_1(\A',n'-1)\rangle}\right)\leq n'.
\end{equation}

Then (\ref{second_exact_seq}) implies
 $${\rm pdim}_T\left(\frac{T}{\mathcal I_1(\A',n'-1):x_1}\right)\leq\max\{n'-1,n'-1\}=n'-1.$$
Back to (\ref{first_exact_seq}) would  finally  give
$${\rm pdim}_T\left(\frac{T}{\mathcal I_1(\A,n-1)}\right)\leq\max\{(n'-1)+1,n'-1\}=n'=n-1,$$
proving the required statement of the theorem.

Thus, it suffices to prove (\ref*{auxiliary_pdim}).

For this, one sets
$\langle x_1,\mathcal I_1(\A',n'-1)\rangle=\langle x_1, x_2y_2-\bar{\ell}_3y_3,\ldots,\bar{\ell}_{n-1}y_{n-1}-\bar{\ell}_ny_n\rangle,$ where we have written $\ell_j=c_jx_1+\bar{\ell}_j$, with $c_j\in\K,  \bar{\ell}_j\in \K[x_2,\ldots,x_k]$, for $3\leq j\leq n$.
Then,
	\begin{equation}\label{pdim_oneplus}
{\rm pdim}_T\left(\frac{T}{\langle x_1,\mathcal I_1(\A',n'-1)\rangle}\right)=1+{\rm pdim}_{T'}\left(\frac{T'}{\langle x_2y_2-\bar{\ell}_3y_3,\ldots,\bar{\ell}_{n-1}y_{n-1}-\bar{\ell}_ny_n\rangle}\right).
	\end{equation}

	Let $\bar{\A}=\{x_2,\bar{\ell}_3,\ldots,\bar{\ell}_n\}$ denote the corresponding stretched arrangement and set
	$$J:=\langle x_2y_2-\bar{\ell}_3y_3,\ldots,\bar{\ell}_{n-1}y_{n-1}-\bar{\ell}_ny_n\rangle\subset T':=\K[x_2,\ldots,x_k;y_2,\ldots,y_n].$$
	{\sc Claim.} ${\rm pdim}_{T'}(T'/J)\leq n'-1$.
	
	If the size of $\bar{\A}$ is $=n-1=n'$ (i.e., no two linear forms of $\bar{\A}$ are proportional), then $J=\mathcal I_1(\bar{\A}, n'-1)$, and by the inductive hypotheses ${\rm pdim}_{T'}(T'/\mathcal I_1(\bar{\A},n'-1))\leq n'-1$.
	
	Otherwise, suppose $s-1\geq 2$ of the linear forms of $\bar{\A}$ are mutually proportional. Without loss of generality, say
	\[
	\bar{\ell}_3=d_3x_2,\ldots,\bar{\ell}_s=d_sx_2,\quad d_i\in\K\setminus\{0\},\; 3\leq i\leq s
	\]
	and
	\[
	\bar{\ell}_j=d_jx_2+L_j,\quad d_j\in \K, 0\neq L_j\in\K[x_3,\ldots,x_k],\, 4\leq j\leq n.
	\]
	
	Then $$J=\langle x_2(y_2-d_3y_3),\ldots,x_2(y_2-d_sy_s),x_2y_2-\bar{\ell}_{s+1}y_{s+1},\ldots,x_2y_2-\bar{\ell}_ny_n\rangle.$$
	
	We now provide the following estimates:
	\begin{enumerate}
		\item[(a)] ${\rm pdim}_{T'}(T'/\langle x_2,J\rangle)\leq 1+(n-s).$
		\item[(b)] ${\rm pdim}_{T'}(T'/\langle J:x_2\rangle)\leq n'-1.$
	\end{enumerate}
	
For (a), note that $\langle x_2,J\rangle=\langle x_2,L_{s+1}y_{s+1},\ldots,L_ny_n\rangle,$ while from the proof of Lemma~\ref{conjecture} we have $${\rm pdim}_{T'}(T'/\langle x_2,J\rangle)\leq 1+(n-s),$$ since $x_2$ is a nonzero divisor in $T'/\langle L_{s+1}y_{s+1},\ldots,L_ny_n\rangle$.
	
\smallskip

As for (b),  we first claim that $J:x_2=\langle y_2-d_3y_3,\ldots,y_2-d_sy_s,x_2y_2-\bar{\ell}_{s+1}y_{s+1},\ldots,x_2y_2-\bar{\ell}_ny_n\rangle$.
The proof is pretty much the same as that of the equality in (\ref{colon_calculation}), hence will be omitted.

	
	
	This equality implies that
	\[
	{\rm pdim}_{T'}\left(\frac{T'}{J:x_2}\right)=s-2+{\rm pdim}_{T''}\left(\frac{T''}{\langle x_2y_2-\bar{\ell}_{s+1}y_{s+1},\ldots,x_2y_2-\bar{\ell}_ny_n\rangle}\right),
	\] where $T'':=\K[x_2,\ldots,x_k;y_2,y_{s+1},\ldots,y_n]$.
	
	Let $\mathcal B:=\{x_2,\bar{\ell}_{s+1},\ldots,\bar{\ell}_n\}$. With same reasoning for $\mathcal B$ as for $\bar{\mathcal A}$ (i.e., removing proportional linear forms), we obtain
	\[
	{\rm pdim}_{T''}\left(\frac{T''}{\langle x_2y_2-\bar{\ell}_{s+1}y_{s+1},\ldots,x_2y_2-\bar{\ell}_ny_n\rangle}\right)\leq (n-s+1)-1=n-s,
	\] and therefore $${\rm pdim}_{T'}\left(\frac{T'}{J:x_2}\right)\leq s-2+n-s=n-2=n'-1.$$
	
	\medskip
	
Drawing on the estimates (a) and (b) above, the	exact sequence of $T'$-modules $$0\longrightarrow T'/(J:x_2)\longrightarrow T'/J\longrightarrow T'/\langle x_2,J\rangle\longrightarrow 0,$$ gives that $${\rm pdim}_{T'}(T'/J)\leq\max\{n'-1,2+n'-s\}\leq n'-1,$$ since $s\geq 3$.
	
Rolling all the way back to (\ref{pdim_oneplus}), we have proved that
$${\rm pdim}_T\left(\frac{T}{\langle x_1,\mathcal I_1(\A',n'-1)\rangle}\right)\leq 1+(n'-1)=n',$$
as required.
\end{proof}

The main result now follows quite smoothly.

\medskip
\newpage

\begin{thm}\label{thm:CM} The Rees algebra of $I_{n-1}(\A)$ is Cohen-Macaulay.
\end{thm}
\begin{proof} From Corollary\ref{Rees_colon} we have the following short exact sequence of $T$-modules
\[
0\longrightarrow\frac{T}{\mathcal I(\A,n-1)}\longrightarrow \frac{T}{\mathcal I_1(\A,n-1)}\longrightarrow \frac{T}{\langle \mathcal I_1(\A,n-1),\ell_1y_1\rangle}\longrightarrow 0.
\]

By Proposition~\ref{proj_symmetric}, the depth of the middle module is $k+1$, while by Lemma~\ref{conjecture} that of the right most module is at least $k$ -- in fact, by the domino effect one has $\langle \mathcal I_1(\A,n-1),\ell_1y_1\rangle=\langle \ell_1y_1,\ell_2y_2,\ldots,\ell_ny_n\rangle$.
By standard knowledge, the depth of the left most module is at least that of the middle module, namely, $k+1$.
\end{proof}

\medskip

A consequence is an alternative proof of a result of Proudfoot-Speyer (\cite{PrSp}):

\begin{cor}
 Let $\A$ be any central arrangement. Then the associated Orlik-Terao algebra  is Cohen-Macaulay.
\end{cor}
\begin{proof}
As we have seen, the Orlik-Terao algebra is the special fiber of the ideal $I_{n-1}(\A)$. Since the latter is a homogeneous ideal generated in one single degree, then the Cohen-Macaulay property of the Rees algebra transfers to its direct summand, the special fiber -- this result has been communicated by W. Vasconcelos and the second author as part of an ongoing joint work.
\end{proof}

\vskip .2in

\renewcommand{\baselinestretch}{1.0}
\small\normalsize 


\bibliographystyle{amsalpha}

\end{document}